\documentclass[10pt]{article}
\usepackage{latexsym,amsfonts,amssymb,amsmath,amsthm}
\usepackage[usenames,dvipsnames]{color}
\begin{document}
\setlength{\baselineskip}{14pt}

\parindent 0.5cm
\evensidemargin 0cm \oddsidemargin 0cm \topmargin 0cm \textheight 22cm \textwidth 16cm \footskip 2cm \headsep
0cm

\newtheorem{theorem}{Theorem}[section]
\newtheorem{lemma}{Lemma}[section]
\newtheorem{proposition}{Proposition}[section]
\newtheorem{definition}{Definition}[section]
\newtheorem{example}{Example}[section]
\newtheorem{corollary}{Corollary}[section]

\newtheorem{remark}{Remark}[section]

\numberwithin{equation}{section}

\def\be{\begin{equation}}
\def\ee{\end{equation}}

\def\p{\partial}
\def\I{\textit}
\def\R{\mathbb R}
\def\C{\mathbb C}
\def\u{\underline}
\def\l{\lambda}
\def\a{\alpha}
\def\O{\Omega}
\def\e{\epsilon}
\def\ls{\lambda^*}
\def\D{\displaystyle}
\def\wyx{ \frac{w(y,t)}{w(x,t)}}
\def\imp{\Rightarrow}
\def\tE{\tilde E}
\def\tX{\tilde X}
\def\tH{\tilde H}
\def\tu{\tilde u}
\def\d{\mathcal D}
\def\aa{\mathcal A}
\def\DH{\mathcal D(\tH)}
\def\bE{\bar E}
\def\bH{\bar H}
\def\M{\mathcal M}
\renewcommand{\labelenumi}{(\arabic{enumi})}
\def\p{\partial}
\def\I{\textit}
\def\R{\mathbb R}
\def\C{\mathbb C}
\def\l{\lambda}
\def\a{\alpha}
\def\O{\Omega}
\def\e{\epsilon}
\def\ls{\lambda^*}
\def\D{\displaystyle}
\def\wyx{ \frac{w(y,t)}{w(x,t)}}
\def\imp{\Rightarrow}
\def\tE{\tilde E}
\def\tX{\tilde X}
\def\tH{\tilde H}
\def\tu{\tilde u}
\def\d{\mathcal D}
\def\aa{\mathcal A}
\def\DH{\mathcal D(\tH)}
\def\bE{\bar E}
\def\bH{\bar H}
\def\M{\mathcal M}
\renewcommand{\labelenumi}{(\arabic{enumi})}

\def\ds{\displaystyle}
\def\undertex#1{$\underline{\hbox{#1}}$}
\def\card{\mathop{\hbox{card}}}
\def\sgn{\mathop{\hbox{sgn}}}
\def\exp{\mathop{\hbox{exp}}}
\def\OFP{(\Omega,{\cal F},\PP)}
\newcommand\JM{Mierczy\'nski}
\newcommand\RR{\ensuremath{\mathbb{R}}}
\newcommand\CC{\ensuremath{\mathbb{C}}}
\newcommand\QQ{\ensuremath{\mathbb{Q}}}
\newcommand\ZZ{\ensuremath{\mathbb{Z}}}
\newcommand\NN{\ensuremath{\mathbb{N}}}
\newcommand\PP{\ensuremath{\mathbb{P}}}
\newcommand\abs[1]{\ensuremath{\lvert#1\rvert}}

\newcommand\normf[1]{\ensuremath{\lVert#1\rVert_{f}}}
\newcommand\normfRb[1]{\ensuremath{\lVert#1\rVert_{f,R_b}}}
\newcommand\normfRbone[1]{\ensuremath{\lVert#1\rVert_{f, R_{b_1}}}}
\newcommand\normfRbtwo[1]{\ensuremath{\lVert#1\rVert_{f,R_{b_2}}}}
\newcommand\normtwo[1]{\ensuremath{\lVert#1\rVert_{2}}}
\newcommand\norminfty[1]{\ensuremath{\lVert#1\rVert_{\infty}}}

\title{Spectral theory for   nonlocal dispersal operators  with time periodic
indefinite weight functions  and applications}

\author{
Wenxian Shen  \\
Department of Mathematics and Statistics\\
Auburn University\\
Auburn, AL 36849, U.S.A. \\\\
Xiaoxia Xie\\
Department of Applied  Mathematics\\
Illinois Institute of Technology\\
Chicago, IL 60616, U.S.A.
}

\date{}
\maketitle

\noindent {\bf Abstract.}
In this paper, we study the spectral theory for  nonlocal dispersal operators  with
time periodic indefinite weight functions subject to  Dirichlet type, Neumann type and spatial periodic type boundary conditions.
 We first  obtain  necessary and sufficient  conditions for the existence of a unique positive principal spectrum point for such operators. We then investigate
upper bounds of principal spectrum points and sufficient conditions for the principal spectrum points
to be principal eigenvalues. Finally we    discuss the applications to nonlinear mathematical models from biology.

\bigskip

\bigskip


\newpage

\section{Introduction}
\setcounter{equation}{0}

Random dispersal operators, such as $u(t, \cdot)\mapsto \Delta u(t, \cdot)$ with proper boundary condition, are often adopted in modeling dissipative systems in applied sciences when the organisms
in the   systems  move randomly between the adjacent spatial locations (see
\cite{ BrCoFl, CaCo, CaCo1, CaCo2,  Co, CoCuPo, Fisher,  Ha, He, He0, HeKo, HessWein, Kolmo, Zhaox}, etc.).
Nonlocal dispersal operators, such
as $u(t, \cdot)\mapsto  \int_{\RR^N}k(y-\cdot)[u(t, y)-u(t, \cdot)]dy$, are applied  in modeling diffusive systems in applied sciences when the systems  exhibit long range internal interactions (see \cite{ FeIsPe,  Fif, GrHiHuMiVi,HuGr, HuMaMiVi,   KaLoSh1, OtDuAl, ShZh0, St}, etc.).   
Diffusive evolution equations with both
 random and  nonlocal dispersals have been widely studied on one hand. On the other hand, there are still many important 
dynamical issues for such systems which are not well understood yet.

  The current paper is devoted to the study of principal spectrum points/principal eigenvalues of  nonlocal  dispersal operators    with time periodic indefinite weight functions and Dirichlet type, Neumann type, and periodic boundary conditions. More precisely, the eigenvalue problem subject to Dirichlet type boundary condition considered in this paper reads as
\begin{equation}
\label{dirichlet-ori}
\begin{cases}
 -\p_t u(t, x)+\int_D \kappa(y-x)u(t, y)dy-u(t, x)+\lambda_1 m_1(t,x)u(t, x)=0,\quad x\in \bar D,\\
u(t+T, x)=u(t, x),
\end{cases}
\end{equation}
where $D\subset \R^N$ is a smooth bounded domain and $m_1(t, x)\neq 0$ is a continuous, time periodic weight function;  the eigenvalue problem subject to  Neumann type boundary condition considered in this paper reads as
\begin{equation}
\label{neumann-ori}
\begin{cases}
 -\p_t u(t, x)+\int_D \kappa(y-x)[u(t, y)-u(t, x)]dy+\lambda_2 m_2(t,x)u(t, x)=0,\quad x\in \bar D,\\
u(t+T, x)=u(t, x),
\end{cases}
\end{equation}
 where $D\in \R^N$ and $m_2(t, x)$ are as  \eqref{dirichlet-ori}; and the eigenvalue problem   subject to  periodic boundary conditions  considered in this paper reads as 
 \begin{equation}
\label{periodic-ori}
\begin{cases}
 -\p_t u(t, x)+\int_{\R^N} \kappa(y-x)[u(t, y)-u(t, x)]dy+\lambda_3 m_3(t,x)u(t, x)=0,\quad x\in \R^N,\\
u(t+T, x)=u(t, x+p_j{\bf e_j})=u(t, x),
\end{cases}
\end{equation}
where $p_j>0$, ${\bf e_j}=(\delta_{j_1}, \delta_{j_2}, \cdots, \delta_{jN})$ ($\delta_{jk}=1$ if $j=k$ and $\delta_{jk}=0$ if $j\neq k$), and $m_3(t, x)$ is a continuous function with $m_3(t+T, x)=m_3(t, x+p_j{\bf e_j})=m_3(t, x)$  for $j=1, 2, \cdots, N$. 
Throughout the paper, we assume that the nonlocal kernel function  $\kappa(\cdot)$  in \eqref{dirichlet-ori}-\eqref{periodic-ori} satisfies
 the following condition
  \be
\label{K}
\kappa(\cdot)\in C_c(\R^N),\,\  \ds\kappa(0)>0,\,\,   \kappa(-z)=\kappa(z), \,\, \text{ and } \int_{\RR^N} \kappa(z)dz=1\tag{K}.
\ee

The eigenvalue problems \eqref{dirichlet-ori}, \eqref{neumann-ori}, and \eqref{periodic-ori} can be viewed as  the nonlocal  dispersal counterparts of the following eigenvalue problems associated to random dispersal operators with time periodic indefinite functions and Dirichlet, Neumann, and periodic boundary conditions,
 \begin{equation}
\label{random-dirichlet-op}
\begin{cases}
 -\p_t u(t, x)+\Delta u(t, x)+\lambda_1m_1(t,x)u(t, x)=0 \, & x\in  D,\\
 u(t, x)=0, & x\in \p D,\\
 u(t, x)=u(t+T, x),
\end{cases}
\end{equation}
 \begin{equation}
\label{random-neumann-op}
\begin{cases}
 -\p_t u(t, x)+\Delta u(t, x)+\lambda_2 m_2(t,x)u(t, x)=0 \, & x\in  D,\\
 \frac{\p u}{\p {\bf n}}(t, x)=0, & x\in \p D,\\
 u(t, x)=u(t+T, x),
\end{cases}
\end{equation}
 and
  \begin{equation}
\label{random-periodic-op}
\begin{cases}
 -\p_t u(t, x)+\Delta u(t, x)+\lambda_3 m_3(t,x)u(t, x) =0\, & x\in  \R^N,\\
 u(t+T, x)=u(t, x+p_j{\bf e_j})=u(t, x),
\end{cases}
\end{equation}
 respectively (see \cite{ShXi1} and references therein  for the relation between nonlocal dispersal operators with Dirichlet type, Neumann type,  and periodic boundary conditions, and random dispersal operators with Dirichlet, Neumann, and periodic boundary conditions).

The eigenvalue problems of random dispersal operators with indefinite weight functions have  been extensively  for more than two decades (see, e.g. \cite{BeHe, Boc, Bo, BrCoFl, BrLi, CaCo1, Co, CoCuPo, He, He0, HeKo, HiKaLa, KaLoYa, LoYa, SeHe} and
references therein). For $i=1$ ($i= 2$, or $i= 3$, respectively), recall that a real number   $\lambda_1^{r,p}$ ($\lambda_2^{r,p}$ or $\lambda_3^{r,p}$, respectively)
 is called a {\it principal eigenvalue} of \eqref{random-dirichlet-op} (\eqref{random-neumann-op}, or \eqref{random-periodic-op}, respectively)
 if
\eqref{random-dirichlet-op} (\eqref{random-neumann-op}, or \eqref{random-periodic-op}, respectively) with
 $\lambda_1=\lambda_1^{r,p}$ ($\lambda_2=\lambda_2^{r,p}$, or  $\lambda_3=\lambda_3^{r,p}$, respectively)  has a positive solution (called {\it  eigenfunction}).
The eigenvalue problem \eqref{random-dirichlet-op}, \eqref{random-neumann-op}, or \eqref{random-periodic-op} are closely related to the following regular  eigenvalue problems,
  \begin{equation}
\label{random-dirichlet-op-eq}
\begin{cases}
 -\p_t u(t, x)+\Delta u(t, x)+\lambda_1m_1(t,x)u(t, x)=\mu_1 u(t, x) \, & x\in  D,\\
 u(t, x)=0, & x\in \p D,\\
 u(t, x)=u(t+T, x),
\end{cases}
\end{equation}
 \begin{equation}
\label{random-neumann-op-eq}
\begin{cases}
 -\p_t u(t, x)+\Delta u(t, x)+\lambda_2 m_2(t,x)u(t, x)=\mu_2 u(t, x) \, & x\in  D,\\
 \frac{\p u}{\p {\bf n}}(t, x)=0, & x\in \p D,\\
 u(t, x)=u(t+T, x),
\end{cases}
\end{equation}
\begin{equation}
\label{random-periodic-op-eq}
\begin{cases}
 -\p_t u(t, x)+\Delta u(t, x)+\lambda_3 m_3(t,x)u(t, x) =\mu_3u(t, x)\, & x\in  \R^N,\\
 u(t+T, x)=u(t, x+p_j{\bf e_j})=u(t, x).
\end{cases}
\end{equation}
For $i=1, 2, 3$, and any given $\lambda_i$, let $\mu^{r, p}_{i}(\lambda_i)$ be the principal eigenvalue of  \eqref{random-dirichlet-op-eq}, \eqref{random-neumann-op-eq}, and \eqref{random-periodic-op-eq}, respectively.
Then $\lambda_1^{r,p}$ (resp. $\lambda_2^{r,p}$, $\lambda_3^{r,p}$) is a principal eigenvalues of
\eqref{random-dirichlet-op} (resp. \eqref{random-neumann-op}, \eqref{random-periodic-op}) if and only if
$\mu^r_1(\lambda_1^{r,p})=0$ (resp. $\mu^r_2(\lambda_2^{r,p})=0$, $\mu^r_3(\lambda_3^{r,p})=0$).

Thanks to its applications in the nonlinear mathematical models, in particular, the population dynamics in biology, the existence of positive principal eigenvalues of  \eqref{random-dirichlet-op}, \eqref{random-neumann-op}, and \eqref{random-periodic-op} is of particular interest and has
been well studied. The time independent version was first studied by Manes and Micheletti in \cite{MaMi}. Then, Hess and Kato   in \cite{HeKo} and Brown and Lin  in \cite{BrLi} obtained some results for   elliptic operators  in Dirichlet boundary case.
% Basically, the following is known: if $m_i$ does not change sign, then \eqref{0} admits one principal eigenvalue. If $m_i$ changes sign, then \eqref{0} admits two principal eigenvalues; one negative and the other positive.
 The proofs of Brown and Lin and Manes and Micheletti are based
on the variational characterization of the principal eigenvalue; the proof of Hess and Kato uses Krein-Rutman's theorem. For more general  elliptic operators subject to various boundary conditions,
we refer to \cite{ Boc, Bo, BrCoFl,  BrLi,  CaCo1,   GoLa, He, HeKo,  Lo-Go0, Lo-Go,  SeHe}  and references therein,  and  for the applications in population dynamics, such as the optimization of spatial arrangement of favorable and unfavorable regions for a species to survive in biological context, we refer to \cite{ CaCo, CaCo2, Co, CoCuPo, HiKaLa,  KaLoYa, LoYa},  etc.

  Among others,  in the time independent case, it is proved that \eqref{random-dirichlet-op} with $m_1(t, x)=m_1(x)$ admits a unique positive principal eigenvalue  if  and only if
\be
\label{random-auto-diri-condi}
m(x_0)>0 \quad \text{ for some } x_0\in D
\ee
(see \cite{GoLa, HeKo, He}),
and \eqref{random-neumann-op} with $m_2(t,x)=m_2(x)$ admits a unique positive principal eigenvalue if and only if
\begin{equation}
\label{random-auto-neum-condi}
m_2(x_0)>0 \quad\text{for some}\,\, x_0\in D\,\,\, \text{and}\,\,\int_D m_2(x)dx<0
\end{equation}
(see \cite{He}).
In the time periodic case,
it is proved that   \eqref{random-dirichlet-op}  has a unique positive principal eigenvalue   if and only if
\be
\label{R-D}
\int_0^T\max_{x\in \bar D}m_1(t, x)dt>0\tag{R-D}
\ee
(see \cite{He0}), and
 \eqref{random-neumann-op} has a unique positive principal eigenvalue   if
\be
\label{R-N}
\int_0^T\max_{x\in \bar D}m_2(t, x)dt>0 \text{ and } \int_D\int_0^Tm_2(t, x)dtdx<0\tag{R-N}
\ee
(see \cite{He0}). There is no result for random dispersal operator subject to the spatial periodic boundary condition so far. But  as we can see in the proofs for nonlocal dispersal operator, results in  spatial periodic boundary case are very similar to those in Neumann boundary case.

The study of \eqref{dirichlet-ori}, \eqref{neumann-ori}, and \eqref{periodic-ori} is of great interest in its own and  will also have important applications
in the study of  many nonlinear mathematical models with nonlocal dispersal from applied science, including
 the following time periodic dispersal evolution equations,
% \be
% \label{n-kpp}
% \begin{cases}
% \p_t u=\int_{D} \kappa(y-x)u(t,y)dy-u(t,x)+\lambda_i uf_i(t,x,u)\quad \text{in } \bar D,\\
% B_{n, b}u(t, x)=0,
% \end{cases}
% \ee
 \begin{equation}
 \label{n-kpp-diri}
 \begin{cases}
 \p_t u=\int_{D} \kappa(y-x)u(t,y)dy-u(t,x)+\lambda_1 uf_1(t,x,u),\quad x\in \bar D,\\
 u(t+T, x)=u(t, x),\\
 u(0, x)\geq 0, u(0, x)\not\equiv 0,
 \end{cases}
 \end{equation}
  \begin{equation}
 \label{n-kpp-neum}
 \begin{cases}
 \p_t u=\int_{D} k(y-x)[u(t,y)-u(t,x)]dy+\lambda_2 uf_2(t,x,u),\quad x\in \bar D,\\
 u(t+T, x)=u(t, x),\\
 u(0, x)\geq 0, u(0, x)\not\equiv 0,
 \end{cases}
 \end{equation}
and
  \begin{equation}
 \label{n-kpp-peri}
\begin{cases}
 \p_t u=\int_{\R^N} k(y-x)u(t,y)dy-u(t,x)+\lambda_3 uf_3(t,x,u),\quad x\in \R^N,\cr
u(t+T, x+p_i{\bf e_i})=u(t, x+p_i{\bf e_i})=u(t, x),\\
 u(0, x)\geq 0, u(0, x)\not\equiv 0,
\end{cases}
 \end{equation}
where $u(0, x)$ is continuous and bounded,  $\lambda_i >0$,  and $f_i(t, x, u)$ satisfies the following condition $(i=1,2,3)$,
 \medskip

 \noindent {\bf (F)} {\it  $f_i$ is $C^1$ in $t\in\RR$ and  $C^3$ in $(x,u)\in \RR^N\times\RR$; $f_i(t, x, u)<0$ for $u\gg 1$ and $\partial _u f_i(t, x, u)<0$ for $u\geq0$};
 $f_i(t+T,x,u)=f_i(t,x,u)$; and when $i=3$, $f_i(t+T, x,u)=f_i(t, x+p_j{\bf e_j}, u)=f(t,x,u)$ for $j=1, 2, \cdots, N$.

  \medskip

%Equations \eqref{n-kpp-diri}, \eqref{n-kpp-neum}, and \eqref{n-kpp-peri} are widely used to model population dynamics of %species. Fisher  in \cite{Fisher} and Kolmogorov et al. in \cite{Kolmo} studied the following special case
%\[
%\p_t u=u_{xx}+u(1-u), \quad x\in \R.
%\]
%Thanks for their  pioneering works, \eqref{n-kpp-diri}, \eqref{n-kpp-neum}, and \eqref{n-kpp-peri} are referred to as Fisher %type or KPP type equations.  Recently, there are many researches towards the KPP type equations, especially for the population %dynamics. We refer to \cite{BaZh, CoDaMa1,  HessWein, Nad1, Zhaox}, etc.  for the random dispersals and \cite{RaSh} for %nonlocal dispersals.

In \eqref{n-kpp-diri}, \eqref{n-kpp-neum}, and \eqref{n-kpp-peri},  $u(t, x)$ represents the density of a species at location $x$ and time $t$, and  $\lambda_i f(t, x, 0) (i=1, 2, 3)$ represents the intrinsic growth rate of a species. Hence only non-negative solutions of \eqref{n-kpp-diri}, \eqref{n-kpp-neum}, and \eqref{n-kpp-peri} are of interest. Notice that $u\equiv 0$ is a solution to  \eqref{n-kpp-diri}, \eqref{n-kpp-neum}, and \eqref{n-kpp-peri}. It is of great interest to know for which $\lambda_1$ (resp. $\lambda_2$, $\lambda_3$),
there is a positive  solution of \eqref{n-kpp-diri} (resp. \eqref{n-kpp-neum},  \eqref{n-kpp-peri}).
As in the random dispersal case,
the spectral theory for the eigenvalue problems \eqref{dirichlet-ori}, \eqref{neumann-ori}, and \eqref{periodic-ori} will play an
important role in the study of positive solutions of \eqref{n-kpp-diri}, \eqref{n-kpp-neum},  and \eqref{n-kpp-peri}).

However, little is known about the eigenvalue problems of
nonlocal dispersal operators with (time periodic) indefinite weight functions.
The  objective of this paper is to investigate the principal spectrum points /eigenvalues (if exists) of the nonlocal time-periodic weighted eigenvalue problems    with Dirichlet type boundary condition \eqref{dirichlet-ori}, Neumann type boundary condition  \eqref{neumann-ori} and spatial periodic type boundary condition  \eqref{periodic-ori}, respectively.
Note that, unlike random dispersal operators, a nonlocal dispersal operator may not have a principal eigenvalue (see \cite{Cov, ShZh0} for examples).   We hence first introduce the notion of the principal spectrum points of \eqref{dirichlet-ori}, \eqref{neumann-ori} and \eqref{periodic-ori}, which are the natural generalization of notion of the principal eigenvalue  of \eqref{random-dirichlet-op}, \eqref{random-neumann-op},  and \eqref{random-periodic-op}. Moreover, due to the lack of compactness of nonlocal operators,  the Krein-Rutman's theorem is not  applicable, and there is in general no variational characterization for time-periodic nonlocal dispersal operators. We overcome the difficulties by developing various new techniques,
 and prove  the necessary and sufficient conditions for the existence of positive principal
spectrum points of  \eqref{dirichlet-ori}, \eqref{neumann-ori},
and \eqref{periodic-ori}, respectively. We also investigate the upper bounds of principal spectrum points and  the sufficient conditions for the principal spectrum points to be principal eigenvalues. As an application, we study asymptotic dynamics of
 \eqref{n-kpp-diri}, \eqref{n-kpp-neum}, and \eqref{n-kpp-peri}  by the spectral theory of weighted nonlocal dispersal operators.

The rest of this paper is organized as follows. In section 2, we  introduce notations, definitions and  state the main results of the paper.  In section 3, we present some preliminary materials to be used in the proof of the main results.  We prove the main results and discuss the application of the  main results in  section 4.

\section{Notations, Definitions, and Main Results}

In this section, we introduce notations, definitions, and state the main results.

We first introduce some standing notations and the concept of principal spectrum points of nonlocal dispersal operators
with time periodic indefinite weight functions.

 Let
\be
\label{space-t12}
\mathcal X_1=\mathcal X_2=\{u\in C(\R\times \bar D, \R)| u(t+T, x)=u(t, x)\}
\ee
with norm $\|u\|_{\mathcal X_i}=\sup_{t\in \R, x\in \bar D}|u(t, x)| (i=1, 2)$,
\be
\label{space-t3}
\mathcal X_3=\{u\in C(\R\times \R^N, \R)|u(t+T, x)=u(t, x+p_j{\bf e_j})=u(t, x) \text{ for } j=1,2,\cdots,N\}
\ee
with norm $\|u\|_{\mathcal X_3}=\sup_{t\in\R, x\in\R^N}|u(t, x)|$. Set
\[
\mathcal X_i^+=\{u\in \mathcal X_i|u\geq 0\}
\]
and
\[
 \mathcal X_i^{++}=\text{Int}{X_i^+}=\{\phi\in \mathcal X_i, \phi>0\}
\]
($i=1, 2, 3$).
 We define the integral operator $K_i$,   the multiplication operator $B_i$, and  the domain $D_i$   $(i=1,2,3)$  as follows,
\begin{equation}
\label{Ki}
K_i: \mathcal X_i\to\mathcal X_i,\,\, K_iu(t, x)=
\begin{cases}
\int_D \kappa(y-x)u(t, y)dy\quad &\forall u\in X_i,\quad i=1,2,\\
\int_{\RR^N}\kappa(y-x)u(t, y)dy\quad&\forall u\in X_3,
\end{cases}
\end{equation}
\be
\label{bi}
B_i: \mathcal X_i\to \mathcal X_i, B_iu=b_i u \,\ \text{ with }\,\,\,
b_i(x)=
\begin{cases}
1 \, & \text{ for } i=1, 3,\\
\int_D\kappa(y-x)dy & \text{ for } i=2,
\end{cases}
\ee
and
\begin{equation}
\label{D-i}
D_i=\begin{cases}  D\quad &{\rm for}\,\, i=1,2,\cr
[0,p_1]\times[0,p_2]\times [0,p_N] &{\rm for}\,\, i=3.
\end{cases}
\end{equation}
For $1\leq i\leq 3$, set
\be
\label{Li}
(\mathcal L_iu)(t, x)=-\p_tu(t, x)+(K_iu)(t, x)-(B_iu)(t, x)
\ee
with domain
\[
\mathcal D (\mathcal L_i)=\{u\in \mathcal X_i| u \text{ is } C^1 \text{ in } t \text{ and } u_t\in \mathcal X_i\}.
\]
Then for $i=1,2,3$, the eigenvalue problems \eqref{dirichlet-ori}, \eqref{neumann-ori}, and \eqref{periodic-ori}  in the space $\mathcal X_i$  can be written uniformly as
\begin{equation}
\label{EP-lambdai}
\mathcal L_iu+\lambda_i m_i u=0,
\end{equation}
Observe that  $K_2=K_1$, $B_3=B_1$ and $D_2=D_1$. The introduction of $K_2$, $B_3$,   and $D_2$ is for convenience.

For any given real number $\lambda_i$, the weighted eigenvalue problem \eqref{EP-lambdai}  is closely related to the  following regular eigenvalue problem on $\mathcal X_i$,
\be
\label{EP-mui}
\mathcal L_i u+\lambda_i m_iu=\mu_i  u.
\ee
 Let  $\sigma_i(\mathcal{L}_i+\lambda_i m_i)$  be the spectrum set of \eqref{EP-mui}
(i.e. the spectrum set of the operator $\mathcal{L}_i+\lambda_im_i$ in $\mathcal{X}_i$) and  $$\ds\mu_i^n(\mathcal{L}_i+\lambda_i m_i)=\sup\{{\rm Re}\mu|\mu\in\sigma_i(\mathcal{L}_i+\lambda_i m_i)\}.
$$

\begin{definition}
\label{PSP-def}
\begin{itemize}
\item[(1)]
$\mu_i^n(\mathcal{L}_i+\lambda_i m_i)$ is called
 {\it the principal spectrum point} of the regular eigenvalue problem  \eqref{EP-mui}.
The principal spectrum point $\mu_i^n(\mathcal{L}_i+\lambda_i m_i)$ is called {\it the principal
eigenvalue} of \eqref{EP-mui}  if \eqref{EP-mui} with $\mu_i=\mu_i^n(\mathcal{L}_i+\lambda_i m_i)$ has a
positive solution in $\mathcal{X}_i$.

\item[(2)]
A real  number $\lambda_i^p(m_i)$ is called
a {\it principal spectrum point}  of the weighted eigenvalue problem  \eqref{EP-lambdai} if $\mu_i^n(\mathcal{L}_i+\lambda_i^p(m_i) m_i)=0$. When $\lambda_i^p(m_i)$ is a principal spectrum
point of \eqref{EP-lambdai} and $\mu_i^n(\mathcal{L}_i+\lambda_i^p(m_i) m_i)$ is a principal eigenvalue of \eqref{EP-mui}, $\lambda_i^p(m_i)$ is also
called a {\it principal eigenvalue} of \eqref{EP-lambdai}.
\end{itemize}
\end{definition}

If not confusion occurs, we may write $\mu_i^n(\mathcal{L}_i+\lambda_i m_i)$ and $\lambda_i^p(m_i)$ as
$\mu_i^n(\lambda_i)$ and $\lambda_i^p$, respectively.

Let
\begin{equation*}
\label{x-d-space}
X_1=X_2=\{u\in C(\bar D, \R)\}
\end{equation*}
with norm $\|u\|_{X_i}=\sup_{x\in\bar D}|u(x)|$ ($i=1,2$),
\begin{equation*}
\label{x-p-space}
X_3=\{u\in C(\RR^N,\RR)\,|\, u(x+p_j{\bf e_j})=u(x),\quad x\in\R^N,\, j=1,2,\cdots,N\}
\end{equation*}
with norm $\|u\|_{X_3}=\max_{x\in\RR^N}|u(x)|$. Set
\begin{equation*}
\label{x-d-positive-cone}
X_i^+=\{u\in X_i\,|\, u\geq 0\},
\end{equation*}
and
\begin{equation*}
\label{x-d-positive-interior}
X_i^{++}=
\begin{cases}
\{u\in X_i^+\,|\, u(x)>0,\quad x\in \bar D\},\quad & i=1,2,\\
\{u\in X_i^+\,|\, u(x)> 0,\quad x\in\RR^N\},\quad & i=3.
\end{cases}
\end{equation*}
In the case that $m_i(t, x)\equiv m_i(x)$, consider
\begin{equation}
\label{autonomous-weight-pev}
K_iu-B_iu+\lambda_i m_iu=0,
\end{equation}
and
\begin{equation}
\label{autonomous-pev}
K_iu-B_iu+\lambda_i m_iu=\mu_i u,
\end{equation}
in $X_i$ ($i=1,2,3$).  Similarly, define
 {\it the principal spectrum  point} $\mu_i^n(\lambda_i)$ of the regular eigenvalue problem
 \eqref{autonomous-pev} to be the largest real part of the spectrum set of \eqref{autonomous-pev}.
We then call a real number  $\lambda_i^p$ {\it the principal spectrum point} of the weighted eigenvalue problem
\eqref{autonomous-weight-pev}
if $\mu_i^n(\lambda_i^p)=0$.

Note that when $m_i(t,x)\equiv m_i(x)$, the principal spectrum point of \eqref{EP-mui} and the principal spectrum point of
\eqref{autonomous-pev} are the same (see Proposition \ref{spectrum-autonomous}).

 Note also that  \eqref{EP-mui}, as well as its time independent version \eqref{autonomous-pev}, is an eigenvalue problem in regular sense for $i=1, 2, 3$, as we can see from Definition \ref{PSP-def} (1).   Many  properties of  the principal spectrum point/principal eigenvalue  have been studied  extensively (see \cite{BaZh, Cov, HuShVi, RaSh, ShXi1, ShZh0} etc.). And we will recall the basic properties and prove some new properties of the principal spectrum point/principal eigenvalue of \eqref{EP-mui} in subsection 3.3.

 However, little is known about the eigenvalue problem \eqref{EP-lambdai}.  We are the first to study the principal spectrum point/principal eigenvalue of nonlocal dispersal operators with time periodic indefinite weight function.  Definition \ref{PSP-def} (2) is a natural generalization of the principal eigenvalue of random dispersal operators with time periodic indefinite weight functions.  And the necessary and sufficient conditions for the existence of positive principal spectrum point of \eqref{EP-lambdai} are in Theorem \ref{PSP-diri}-\ref{PSP-peri}, and the properties of  the  principal spectrum point of \eqref{EP-lambdai}  are in Theorem \ref{upper-bounds}.

For $i=1, 2, 3$, let
\be
\label{hat-mi}
\widehat m_i(x)=\frac{1}{T}\int_0^T m_i(t,x)dt,
\ee
\be
\label{mi-max-min}
\widehat m_{i, \max}=\max_{x\in \bar D_i}m_i(x), \text{ and } \, \widehat m_{i, \min}=\min_{x\in \bar D_i}m_i(x),
\ee
 and
\be
\label{P}
\mathcal P(m_i)=\int_0^T\max_{x\in\bar D_i}m_i(t, x)dt.
\ee

We now state the main results of the paper.
We first  state the main results on  the  necessary and sufficient conditions for the existence and uniqueness of positive
principal spectrum points of \eqref{dirichlet-ori}, \eqref{neumann-ori}, and \eqref{periodic-ori}.

 \begin{theorem}[Necessary and sufficient  condition in the Dirichlet boundary case]
 \label{PSP-diri}
Suppose $\kappa(\cdot)$ satisfies (K) and $m_1\in\mathcal X_1$.  The eigenvalue problem \eqref{dirichlet-ori} has exactly one positive principal spectrum point, denoted by $\lambda_1^p$, if and only if
\be
\label{D}
\mathcal P(m_1)>0\tag{D}.
\ee
 More precisely, we show the following.
 \begin{itemize}
\item[(1)] There is $\lambda_1^p>0$ such that $\mu_1^n(\lambda_1^p)=0$ if and only if
(D) holds.
\item[(2)] If $\lambda_1^{p, 1},\lambda_1^{p, 2}>0$ are such that $\mu_1^n(\lambda_1^{p, 1})=\mu_1^n(\lambda_1^{p, 2})=0$, then $\lambda_1^{p, 1}=\lambda_1^{p, 2}$.
\end{itemize}
\end{theorem}

\begin{corollary}
\label{PSP-diri-coro}
$\quad$
\begin{itemize}
\item[(1)] If $m_1(t, x)\equiv m_1(t)$, then \eqref{dirichlet-ori} has exactly one positive principal spectrum point $\lambda_1^p$  if and only if
$$
\int_0^T m_1(t)dt>0.
$$
Moreover,  $\lambda_1^p$ is a principal eigenvalue of \eqref{dirichlet-ori} and $\ds \lambda_1^p=\frac{-\bar\lambda_1}{\widehat m_1}$, where
 $\bar \lambda_1(<0)$ is the principal eigenvalue of $K_1-B_1(=K_1-\mathcal I)$.

\item[(2)] If $m_1(t, x)\equiv m_1(x)$ , consider
\begin{equation}
\label{autonomous-dirichlet-ori}
\left[\int_D\kappa(y-x)dy-u(x)\right]+\lambda_1m_1(x)u(x)=0,\quad x\in\bar D.
\end{equation}
There exists a unique positive principal spectrum point $\lambda_1^p$ of  \eqref{autonomous-dirichlet-ori}
if and only if $m_1(x_0)> 0$ for $x_0\in  D_1$.
\end{itemize}
\end{corollary}

\begin{remark}
\label{PSP-dirichlet-remark} Theorem \ref{PSP-diri} and Corollary \ref{PSP-diri-coro}
extend the principal eigenvalue theory for random dispersal operators with time independent or time periodic indefinite weight
functions subject to Dirichlet  boundary condition to nonlocal dispersal operators  with time independent or time periodic indefinite weight
functions subject to Dirichlet type  boundary condition.
\end{remark}

\begin{theorem}[Necessary and sufficient  condition in the Neumann boundary case]
\label{PSP-neum}
Suppose that $\kappa(\cdot)$ satisfies (K),    $m_2(\cdot, \cdot)\in\mathcal X_2$ and $m_2(t, x)\not \equiv m_2(t)$.  The eigenvalue problem \eqref{neumann-ori} has exactly one positive principal spectrum point, denoted by $\lambda_2^p$, if and only if
\be
\label{N}
\mathcal P (m_2)>0 \text{   and  } \int_{D_2} \int_0^Tm_2(t, x)dtdx<0.\tag{N}
\ee More precisely, we show the following.
\begin{itemize}
\item[(1)]  There is a  $\lambda_2^p>0$ such that $\mu_2^n(\lambda_2^p)=0$ if and only if (N) holds.

\item[(2)] If $\lambda_2^{p, 1}, \lambda_2^{p, 2}>0$ are such that $\mu_2^n(\lambda_2^{p, 1})=\mu_2^n(\lambda_2^{p, 2})=0$, then
$\lambda_2^{p, 1}=\lambda_2^{p, 2}$.
\end{itemize}
\end{theorem}

\begin{corollary}
\label{PSP-neum-coro}
\begin{itemize}
\item[(1)]  If $m_2(t,x)\equiv m_2(t)$, then
$$\mu_2^n(\lambda_2)=\lambda_2\widehat m_2.
$$
It then follows that, if $\int_0^T m_2(t)dt\not =0$, then  there is no positive principal spectrum point of \eqref{neumann-ori},
and if $\int_0^T m_2(t)dt=0$, then any $\lambda_2>0$ is a positive principal spectrum point  of \eqref{neumann-ori}.

\item[(2)]
 If $m_2(t, x)\equiv m_2(x)$, consider
\begin{equation}
\label{autonomous-neumann-ori}
\int_D\kappa(y-x)[u(y)-u(x)]dy+\lambda_2m_2(x)u(x)=0,\quad x\in\bar D.
\end{equation}
There is exactly one positive principal spectrum point $\lambda_2^p$ of \eqref{autonomous-neumann-ori}
 if and only if
$$
m_2(x_0)>0\quad \text{for some}\,\, x_0\in D_2\quad {\rm and}\,\,\int_{D_2} m_2(x)dx<0.
$$
\end{itemize}
\end{corollary}

\begin{remark}
\label{PSP-neum-remark}
 Theorem \ref{PSP-neum} and Corollary \ref{PSP-neum-coro}
extend the principal eigenvalue theory for random dispersal operators with time independent or  time periodic indefinite weight
functions subject to Neumann boundary condition to nonlocal dispersal operators with time independent or time periodic indefinite weight functions subject to Neumann type  boundary condition.
\end{remark}

\begin{theorem}[Necessary  and sufficient condition in the periodic boundary case]
\label{PSP-peri}
 Suppose that $\kappa(\cdot)$ satisfies (K),  $m_3(\cdot, \cdot)\in\mathcal X_3$, and $m_3(t, x)\not \equiv m_3(t)$ for $t\in \R$ and $x\in \R^N$. The eigenvalue problem \eqref{periodic-ori} has exactly one positive principal spectrum point, denoted by $\lambda_3^p$, if and only if
 \be
\label{P}
\mathcal P (m_3)>0 \text{   and  } \int_{D_3} \int_0^Tm_3(t, x)dtdx<0.\tag{P}
\ee
More precisely, we show the following.
\begin{itemize}
\item[(1)]  There is a  $\lambda_3^p>0$ such that $\mu_3^n(\lambda_3^p)=0$ if and only if
(P) holds.

\item[(2)] If $\lambda_3^{p, 1},\lambda_3^{p, 2}>0$ are such that $\mu_3^n(\lambda_3^{p, 1})=\mu_3^n(\lambda_3^{p, 2})=0$, then
$\lambda_3^{p, 1}=\lambda_3^{p, 2}$.
\end{itemize}
\end{theorem}

\begin{corollary}
\label{PSP-peri-coro}
\begin{itemize}
\item[(1)]  If $m_3(t, x)=m_3(t)$,
 then
$$\mu_3^n(\lambda_3)=\lambda_3 \widehat m_3.
$$
It then follows that, if $\int_0^T m_3(t)dt\not =0$, then  there is no positive principal spectrum point of \eqref{periodic-ori},
and if $\int_0^T m_3(t)dt=0$, then any $\lambda_3>0$ is a positive principal spectrum point of \eqref{periodic-ori}.

\item[(2)]
If $m_3(t, x)\equiv m_3(x)$, consider
\begin{equation}
\label{autonomous-periodic-ori}
\int_{\R^N}\kappa(y-x)[u(y)-u(x)]dy+\lambda_3m_3(x)u(x)=0.
\end{equation}
There is exactly one positive principal spectrum point $\lambda_3^p$ of \eqref{autonomous-periodic-ori}
 if and only if
$$
m_3(x_0)>0\quad \text{for some}\,\, x_0\in D_3,\quad {\rm and}\,\,\int_{D_3}  m_3(x)dx<0.
$$
 \end{itemize}
\end{corollary}

Next, we state the main result  on the upper bounds  of principal spectrum points of \eqref{dirichlet-ori},
\eqref{neumann-ori}, and \eqref{periodic-ori},   and sufficient conditions for the principal spectrum points of \eqref{dirichlet-ori},
\eqref{neumann-ori}, and 
\eqref{periodic-ori}
to be principal eigenvalues.

Consider the eigenvalue problem with indefinite weight function
$\hat m_i$,
\begin{equation}
\label{average-weight-pev}
K_iu-B_iu+\lambda_i \widehat m_i u=0,
\end{equation}
and the regular eigenvalue problem
% \eqref{autonomous-pev}
\begin{equation}
\label{average-pev}
K_iu-B_iu+\lambda_i\widehat m_i u=\mu_i(\lambda_i) u
\end{equation}
in $X_i$.

\begin{theorem}[Properties of the  principal spectrum points]
\label{upper-bounds}
$\quad$
\begin{itemize}
\item[(1)] (Upper bounds) If \eqref{average-weight-pev} has a unique positive principal spectrum point $\lambda_i^p(\widehat m_i)$, then \eqref{EP-lambdai}
has also a unique positive principal spectrum point $\lambda_i^p(m_i)$ and
$$
\lambda_i^p(m_i)\le \lambda_i^p(\widehat m_i).
$$

\item[(2)] (Sufficient conditions of the existence of positive principal eigenvalues)
\item[(i)] Assume that $N=1$ or $2$,   $b_i+\widehat m_i$ is  $C^N$,
and  \eqref{EP-lambdai}
has a unique positive principal spectrum point $\lambda_i^p$, then $\lambda_i^p$ is a positive principal eigenvalue of
\eqref{EP-lambdai}.

\item[(ii)] For $N\geq 1$, and $i=1$ or $3$, assume that  $\ds\int_{D_i}\frac{1}{\widehat m_{i,\max} -\widehat m_i(x)}dx=\infty$ and  \eqref{EP-lambdai}
has a unique positive principal spectrum point $\lambda_i^p$, then $\lambda_i^p$ is a positive principal eigenvalue of
\eqref{EP-lambdai}.
\end{itemize}
\end{theorem}

\begin{remark}
\label{average-pev-remark}
 Assume that $\widehat m_i(x_0)>0$ for some $x_0\in D_i$,  and in addition,
$\int_{D_i}\widehat m_i(x)dx<0$ in the case $i=2,3$. Then \eqref{average-weight-pev} has a unique positive principal spectrum point $\lambda_i^p(\widehat m_i)$.
\end{remark}

%Further, the  lower bound of the principal spectrum of \eqref{EP-mui} was first considered by Rawal and Shen. In fact, it is proved %in  \cite[Theorem C]{RaSh} that
%\[
%\mu_i^n(\lambda_i)\geq \widehat\mu_i^n(\lambda_i),
%\]
%where $\widehat\mu_i^n$ is the principal spectrum point of  \eqref{EP-mui}  with $m_i(t, x)$ being replaced by $\widehat %m_i(x)$.  Hence, if the principal spectrum point of \eqref{dirichlet-ori}, \eqref{neumann-ori}, or \eqref{periodic-ori} exists, it has %an upper bound.
%\begin{corollary}[Upper bound of the principal spectrum point]
%\label{property}
%For $i=1, 2, 3$, $m_i\in \mathcal X_i$. If the positive principal spectrum point $\lambda_i^p(m_i)$ exists, we have
 %\[
%\lambda_i^p\leq \widehat \lambda_i^p
%\]
%where $\widehat\lambda_i^p$ is the principal spectrum point of \eqref{dirichlet-ori}, \eqref{neumann-ori}, or \eqref{periodic-ori} %with $m_i(t, x)$ being replaced by $ \widehat m_i(x)$ (see \eqref{average-mi}).
%\end{corollary}

Finally, we consider the applications of the principal spectrum point theory of nonlocal dispersal operators with
time periodic indefinite weight function to the KPP type equations \eqref{n-kpp-diri}, \eqref{n-kpp-neum}, and \eqref{n-kpp-peri}. And  we  prove  the following result.

\begin{theorem}
\label{KPP}
Assume (K) and (F) hold.  Denote  the principal spectrum point of the eigenvalue problem \eqref{EP-mui} with $m_i(t, x)= f_i(t, x, 0)$ by $\mu_i^n(\lambda_i)$. Assume in addition that $f_1(t, x, 0)$, $ f_2(t, x, 0)$, and $ f_3(t, x, 0)$ satisfies (D), (N), (P), respectively.
  Then \eqref{n-kpp-diri} (\eqref{n-kpp-neum} or \eqref{n-kpp-peri}) admits a unique positive time-periodic solution $u^*(t, x)$ if  and only if
$$
\lambda_i>\lambda_i^p,
$$
where $\lambda_i^p$ is the positive number such that $\mu_i^n(\lambda_i^p)=0$.
\end{theorem}

\section{Preliminary}

In this section, we present some notations and preliminary materials to be used in the proofs of the main results in Section 4.
We first present a technical lemma in subsection 3.1. Then we present a comparison principle for solutions of some
linear nonlocal evolution equations in subsection 3.2. Finally we present some basic properties of principal spectrum points of
the regular eigenvalue problem
\eqref{EP-mui} in subsection 3.3. Throughout this section, 
for $1\leq i\leq 3$,  let
\be
\label{hi}
h_i(t, x; \lambda_i)=-b_i(x)+\lambda_i m_i(t, x),
\ee
and
\be
\label{hathi}
\widehat h_i(x;\lambda_i)=-b_i(x)+\lambda_i\widehat m_i(x),
\ee
where $\widehat m_i$ is as in \eqref{hat-mi}. For $D_i$ defined in  \eqref{D-i}, set
\be
\label{hathi-min-max}
\widehat h_{i, \max}(\lambda_i)=\max_{x\in \bar D_i}  \widehat h_i(x;\lambda_i), \quad \text{ and } \widehat h_{i, \min}(\lambda_i)=\min_{x\in \bar D_i}\widehat h_i(x;\lambda_i).
\ee
If there is no confusion, we  may omit $\lambda_i$, and abbreviate the notations in \eqref{hi}, \eqref{hathi}, and \eqref{hathi-min-max}  as $h_i(t, x)$, $\widehat h_i(x)$ and $\widehat h_{i, \max}$, respectively.

\subsection{A technical lemma}

In this subsection, we provide a useful technical lemma.

\begin{lemma}
\label{technical-lm}
Let $1\le i\le 3$. For any fixed $0<\lambda_i\in \R$, any $m_i\in \mathcal X_i$ and any $\epsilon>0$, there is $m_{i,\epsilon}\in\mathcal X_i$ satisfying that
$$
\|m_i-m_{i,\epsilon}\|_{\mathcal X_i}<\epsilon,
$$
 %$h_i^\epsilon(\cdot)$ is in  $C^N$ where
 $b_i(x)+\lambda_i\widehat m_{i, \epsilon}(x)$  is $C^N$, $b_i(x)+\lambda_i\widehat m_{i, \epsilon}(x)$  attains its maximum at some $x_0\in {\rm Int}(D_i)$, and the derivatives of $b_i+\lambda_i\widehat m_{i, \epsilon}(x)$
 up to order $N-1$ at
$x_0$ are zero.
\end{lemma}

\begin{proof} It follows from \cite[Lemma 4.1]{RaSh}. For the self-completeness, we provide a proof in the following.

 We prove the case  $i=1$ or $2$. The case $i=3$ can be proved similarly. Without loss of generality, we assume $\lambda_i=1$, and recall that $\widehat h_i(x)=-b_i(x)+\widehat m_i(x)$ (see \eqref{hathi}).

 First, let $\tilde x_0\in \bar D_i$ be such that
$$
\widehat h_i(\tilde x_0)=\max_{x\in\bar D_i}  \widehat h_i(x).
$$
For any $\epsilon>0$, there is $\tilde x_\epsilon \in {\rm Int}(D_i)$  such that
\begin{equation}
\label{eq1}
\widehat h_i(\tilde x_0)-\widehat h_i(\tilde x_\epsilon)<\frac{\epsilon}{3}.
\end{equation}
Let $\tilde\sigma>0$ be such that
$$
B(\tilde x_\epsilon,\tilde\sigma)\Subset D_i,
$$
where $B(\tilde x_\epsilon,\tilde\sigma)$ denotes the open ball with center $\tilde x_\epsilon$ and radius
$\tilde\sigma$.

Note that there is $\xi_i(\cdot)\in C(\bar D_i)$  such that $0\leq \xi_i(x)\leq 1$, $ \xi_i(\tilde x_\epsilon)=1$, and
${\rm supp}( \xi_i)\subset B(\tilde x_\epsilon,\tilde\sigma)$. Let
\[
\overline m_{i, \epsilon}(t, x)=m_i(t, x)+\frac{\epsilon}{3} \xi_i(x),
\]
and
\begin{equation}
\label{eq0}
 \overline h_{i,\epsilon}(x)=-b_i(x)+\widehat m_{i}(x)+\frac{\epsilon}{3} \xi_i(x).
\end{equation}
Then  $\overline m_{i, \epsilon}(t, \cdot)$ and $\overline h_{i,\epsilon}(\cdot)$ is continuous on $\bar D_i$,
\be
\label{overline-ai}
\|\overline m_{i, \epsilon}(t, \cdot)-m_i(t, \cdot)\|\leq \frac{\epsilon}{3}
\ee
 and  $\overline h_{i,\epsilon}(\cdot)$ attains its maximum in ${\rm Int}(D_i)$.

Let $\widetilde D_i\subset \RR^N$ be such that $D_i\Subset \widetilde D_i$. Note that $\overline h_{i,\epsilon}(\cdot)$ can be continuously extended to $\widetilde D_i$.
Without loss of generality, we may then assume that $\overline h_{i,\epsilon}(\cdot)$ is a continuous function on $\tilde D_i$ and there is
$x_0\in {\rm Int}(D_i)$  such that
$\overline h_{i,\epsilon}(x_0)=\sup_{x\in\widetilde D_i}\overline h_{i,\epsilon}(x)$.
Observe that there is $\sigma>0$ and $\widetilde h_{i,\epsilon}(\cdot)\in C(\widetilde D_i)$  such that
$B(x_0,\sigma)\Subset D_i$,
\begin{equation}
\label{eq3}
0\leq\widetilde h_{i,\epsilon}(x)-\overline h_{i,\epsilon}(x)\le\frac{ \epsilon}{3}\quad \forall \,  x\in\widetilde D_i,
\end{equation}
and
$$
\widetilde h_{i,\epsilon}(x)=\widehat h_{i,\epsilon}(x_0)\quad \forall \, x\in B(x_0,\sigma).
$$

Let
$$
\eta(x)=\begin{cases} C\exp(\frac{1}{\|x\|^2-1})\quad &{\rm if}\,\ \|x\|<1,\cr\cr
0\quad &{\rm if}\,\ \|x\|\geq 1,
\end{cases}
$$
where $C>0$ is such that $\int_{\RR^N}\eta(x)dx=1$.
For given $\delta>0$, set
$$
\eta_\delta(x)=\frac{1}{\delta^N}\eta\left(\frac{x}{\delta}\right).
$$
Let
$$
h_{i,\epsilon,\delta}(x)=\int_{\widetilde D_i}\eta_\delta(y-x)\widetilde h_{i,\epsilon}(y)dy.
$$
By \cite[Theorem 6, Appendix C]{Ev}, $ h_{i,\epsilon,\delta}(\cdot)$ is in $C^\infty(\widetilde D_i)$ and when $0<\delta\ll 1$,
\begin{equation}
\label{eq4}
|h_{i,\epsilon,\delta}(x)-\widetilde h_{i,\epsilon}(x)|<\frac{\epsilon}{3}\quad \forall \, x\in\bar D_i.
\end{equation}
It is not difficulty to see that for $0<\delta\ll 1$,
$$
h_{i,\epsilon,\delta}(x)\leq
\widetilde h_{i,\epsilon}(x_0)\quad \forall x\in B(x_0,\sigma),
$$
and
$$
h_{i,\epsilon,\delta}(x)=\widetilde h_{i,\epsilon}(x_0)\quad
\forall x\in B(x_0,\sigma/2).
$$
Fix $0<\delta\ll 1$. Let
$$\widehat h_{i, \epsilon}(x)=h_{i,\epsilon,\delta}(x).
$$
Then
$\widehat h_{i, \epsilon}(\cdot)$ attains its maximum at some $x_0\in{\rm Int}(D_i)$,
and the partial derivatives of $\widehat h_{i,\epsilon}(\cdot)$ up to order $N-1$ at $x_0$ are zero.
Let
$$
 m_{i, \epsilon}(t, x)= \overline m_{i, \epsilon}(t, x)   +\widehat h_{i, \epsilon}(x)-\overline h_{i,\epsilon}(x).
$$
Then $m_{i, \epsilon}\in{\mathcal X}_2$,
$$
\|m_i-m_{2,\epsilon}\|\leq \|m_i-\overline m_{i, \epsilon}\|+\|\widehat h_{i, \epsilon}-\overline h_{i, \epsilon}\|\leq\|m_i-\overline m_{i, \epsilon}\|+\|\widehat h_{i, \epsilon}-\widetilde h_{i, \epsilon}\|+\|\widetilde h_{i, \epsilon}-\overline h_{i, \epsilon}\|\leq \epsilon,
$$
and
\[
-b_i(x)+\widehat m_{i,\epsilon}(x)=\widehat h_{i, \epsilon}(x).
\]
Therefore, $-b_i+\widehat m_{i, \epsilon}$ is $C^N$, attains its maximum at some point $x_0\in {\rm Int}D_i$, and the partial derivatives of $-b_i+\widehat m_{i, \epsilon}$ up to order $N-1$ at $x_0$ are zero. The lemma is thus proved.
\end{proof}

\subsection{Comparison Principle} In this subsection, we  present the comparison principle for
 the solutions to the following evolution equations associated to the eigenvalue problem \eqref{EP-mui} with $\lambda_i\ge 0$,
\begin{equation}
\label{dirichlet-ori-evol}
 \p_t u(t, x)=\int_D \kappa(y-x)u(t, y)dy-u(t, x)+ \lambda_1 m_1(t,x)u(t, x) \quad \text{ in } \bar D,
\end{equation}
\begin{equation}
\label{neumann-ori-evol}
 \p_t u(t, x)=\int_D \kappa(y-x)[u(t, y)-u(t, x)]dy+ \lambda_2 m_2(t,x)u(t, x) \quad \text{ in } \bar D,
\end{equation}
and
 \begin{equation}
\label{periodic-ori-evol}
 \p_t u(t, x)=\int_{\R^N} \kappa(y-x)[u(t, y)-u(t, x)]dy+\lambda_3 m_3(t, x)u(t, x) \quad \text{ in } \R^N,
\end{equation}
where $m_i(t, x)\in \mathcal X_i(i=1, 2, 3)$.

%We use $u_1(t, x; s, u_0)$ (resp. $u_2(t, x; s, u_0)$, $u_3(t, x; s, u_0)$) denotes the solution of \eqref{dirichlet-ori-evol} (resp. \eqref{neumann-ori-evol}, \eqref{periodic-ori-evol}) with $u_1(s, \cdot; s, u_0)=u_0(x)\in X_1$ (resp.  $u_2(s, \cdot; s, u_0)=u_0(x)\in X_2$,  $u_3(s, \cdot; s, u_0)=u_0(x)\in X_3$).
By general semigroup theory,  \eqref{dirichlet-ori-evol} (resp. \eqref{neumann-ori-evol}, \eqref{periodic-ori-evol})  generates evolution families $\{\Phi_1(t, s; m_1)\}$ (resp. $\{\Phi_2(t, s; m_2), \Phi_3(t, s; m_3)\}$) on $X_1$ (resp. $X_2$, $X_3$), that is, for any $u_0\in X_1$ (resp. $u_0\in X_2, u_0\in X_3$), $u(t, x; s, u_0, m_1):=(\Phi_1(t, s; m_1)u_0)(x)$ (resp. $u(t, x; s, u_0, m_2):=(\Phi_2(t, s; m_2)u_0)(x)$, $u(t, x; s, u_0, m_3):=(\Phi_3(t, s; m_3) u_0)(x)$) is the unique  solution of \eqref{dirichlet-ori-evol} (resp. \eqref{neumann-ori-evol}, \eqref{periodic-ori-evol}) with $u(s, x; s, u_0, m_1)=u_0(x)$ (resp. $u(s, x; s, u_0, m_2)=u_0(x)$, $u(s, x; s, u_0, m_3)=u_0(x)$).

\begin{definition}
\label{super-sub-sol}
A bounded  measurable   function $u(t, x)$ on $[0, T)\times \bar D$ is called a {\it super-solution} (or {\it sub-solution}) of \eqref{dirichlet-ori-evol} if for any $x\in \bar D$, $u(t, x)$ is differentiable for all  but finite many $t$'s  in  $[0, T)$ and satisfies that for each $x\in \bar D$,
\[
\p_t u\geq (\text{ or } \leq) \int_D \kappa(y-x)u(t, y)dy-u(t, x)+ m_1(t,x)u(t, x)
\]
for all but finite many $t$'s in  $[0, T)$.
\end{definition}
{\it Super-solutions and sub-solutions} of \eqref{neumann-ori-evol} and $\eqref{periodic-ori-evol}$ are defined in an analogous way.
\begin{proposition}[Comparison principle]
\label{comparison}
$\quad$
\begin{itemize}
\item[(1)] If $u^1(t, x)$ and $u^2(t, x)$ are bounded sub- and super- solution of \eqref{dirichlet-ori-evol} (resp. \eqref{neumann-ori-evol}, \eqref{periodic-ori-evol}) on $[s, T)$, respectively, and $u^1(s, \cdot)\leq u^2(s, \cdot)$, then $u^1(t, \cdot)\leq u^2(t, \cdot)$ for $t\in [s, T)$.
\item[(2)] If $u^1, u^2\in X_i$, $u^1\leq u^2$ and $u^1\neq u^2$, then
 \[\Phi_i(t, s; m_i) u^1\ll \Phi_i(t, s; m_i) u^2 \text{ for all } t>s.
 \]
\item[(3)] If $u_0\in X_i^+$, and $m_i^1, m_i^2\in \mathcal X_i$, if $m_i^1\leq m_i^2$, then
\[
\Phi_i(t, s; m_i^1)u_0\leq \Phi_i(t, s; m_i^2)u_0 \text{ for all } t>s.
\]
\end{itemize}
%(2) For every $u_0\in X_i^+$, $u_i(t, x; s, u_0)$ exists for all $t\geq s$.
\end{proposition}
\begin{proof}
(1)  follows from the arguments in \cite[Proposition 3.1 (1)]{ShZh0}.

(2) follows from the arguments in \cite[Proposition 3.1 (2)]{ShZh0}.

(3) We consider the case $i=1$. Other cases ca be proved similarly.

Note that $u_1(t, x; s, u_0, m_1^2)$ is a supersolution of \eqref{dirichlet-ori-evol} with $m_1$ being replaced by $m_1^1$. Then by (1),
\[
u_1(t, \cdot; s, u_0, m_1^1)\leq  u_1(t, \cdot; s, u_0, m_1^2)  \text{ for all } t>s.
\]
\end{proof}
For simplicity in notation, put $\Phi_i(T; m_i)=\Phi_i(T, 0; m_i) (i=1, 2, 3)$, and let $r(\Phi_i(T;m_i))$ be the spectral radius of $\Phi_i(T, 0; m_i)$.

\subsection{Basic properties of principal spectrum points/principal eigenvalues}
Our objective in this subsection is to  study some basic properties of  the principal spectrum point of the regular eigenvalue problem \eqref{EP-mui}.

Recall that  in \eqref{EP-mui}, $\mathcal (\mathcal L_i+\lambda_i m_i)u=-\p_t u-B_i u+K_iu+\lambda_i m_i u$,
and we may use  $\mu_i^n(\mathcal L_i+\lambda_i m_i)$ or $\mu_i^n(\lambda_i)$ to denote the principal  spectrum point of \eqref{EP-mui}, if no confusion occurs.
 And when $\lambda_i=0$, we use $\mu_i^n(0)$  to denote the principal  $\mathcal L_i$.

\begin{proposition}
\label{spectrum-sol-operator}
For $1\leq i\leq 3$, $\mu_i^n (\mathcal L_i+\lambda_i m_i)=\frac{\ln r(\Phi_i(T; m_i))}{T}$.
\end{proposition}
\begin{proof}
It follows from \cite[Proposition 3.10]{RaSh}.
\end{proof}

\begin{proposition}
\label{spectrum-limit} Let $1\le i\le 3$.
Given any $u_0\in X_i^+\setminus\{0\}$,
$$
\mu_i^n(\mathcal{L}_i+\lambda_i m_i)=\lim_{t\to\infty}\frac{\ln\|\Phi_i(t,0;m_i)u_0\|}{t}.
$$
\end{proposition}

\begin{proof}It follows from the arguments in \cite[Proposition 2.5 and Theorem 3.2]{HuShVi}.
\end{proof}

\begin{proposition}
\label{spectrum-autonomous}
Let $1\le i\le 3$.
\begin{itemize}
\item[(1)]
If $m_i(t,x)\equiv m_i(x)$, then the principal spectrum point of \eqref{EP-mui} equals to the principal spectrum point of
\eqref{autonomous-pev}.

\item[(2)] If $m_i(t,x)\equiv m_i(t)$, then $\mu_i^n(\mathcal{L}_i+\lambda_i m_i)=\mu_i^n(0)+\lambda_i  \widehat m_i$.
\end{itemize}
\end{proposition}

\begin{proof}
(1)  It follows from \cite[Proposition 3.3]{ShXi1} and Proposition \ref{spectrum-limit}.

(2) It follows from Proposition \ref{spectrum-limit} and the fact that
$$
\Phi_i(t,0;m_i)=e^{\lambda_i \int_0^t m_i(s)ds}\Phi_i(t,0;0).
$$
\end{proof}

 Recall that $h_i(t, x)$ is defined in \eqref{hi}. Set
\be
\label{hi(x)}
\mathcal H_i: \mathcal{D}(\mathcal {H}_i)\to \mathcal X_i, (\mathcal H_i u)(t, x)=-\p_t u(t, x)+h_i(t, x)u(t, x)
\ee
with $\mathcal D(\mathcal H_i)=\mathcal D(\mathcal L_i)$. Note that $\mathcal H_i+K_i=\mathcal L_i+\lambda_i m_i$. Hence, we may use
 $\mu_i^n(\mathcal H_i+K_i)$ or equivalently $\mu_i^n(\mathcal L_i+\lambda_i m_i)$ in different situations to denote the principal spectrum point of the eigenvalue problem \eqref{EP-mui}.  We use $\mathcal Re\{\sigma(\mathcal H_i)\}$  to denote the real part of the spectrum set of $\mathcal{H}_i$.

\begin{proposition}
\label{spectrum-Hi}
Let $1\le i\le 3$.
\begin{itemize}
\item[(1)]
 $\mathcal Re\{\sigma(\mathcal H_i)\}=[\widehat h_{i,\min},\widehat h_{i,\max}]$.

\item[(2)] $\mu_i^n(\mathcal{L}_i+\lambda_i m_i)\ge \mu_i^n(\mathcal{L}_i+\lambda_i \widehat m_i)\ge \widehat h_{i,\max}$.

\item[(3)] $\mu_i^n(\mathcal{L}_i+\lambda_i m_i)$ is the principal eigenvalue if and only if $\mu_i^n(\mathcal{L}_i+\lambda_i m_i)>
\widehat h_{i,\max}$.
\end{itemize}
\end{proposition}
\begin{proof}
(1) It follows from Lemma 3.7 in \cite{HuShVi}.

(2) $\mu_i^n(\mathcal{L}_i+\lambda_i m_i)\ge \mu_i^n(\mathcal{L}_i+\lambda_i \widehat m_i)$ follows from \cite[Theorem C]{RaSh}, and $\mu_i^n(\mathcal{L}_i+\lambda_i \widehat m_i)\ge \widehat h_{i,\max}$ follows from Proposition \ref{spectrum-autonomous} (1) and \cite[Proposition 3.9]{ShXi0}.

(3) It follows from \cite[Theorem A]{RaSh}.
\end{proof}

For fixed $\lambda_i$, let (S1), (S2), and (S3) be the following standing assumptions.

 \smallskip
\noindent {\bf (S1)}  {\it For $1\leq i\leq 3$,
$\widehat h_i(\cdot;\lambda_i)$ is  $C^N$, there is some $x_0\in {\rm{Int}} D_i$ in the case $i=1, 2$ and $x_0\in D_3$ in the case of $i=3$ satisfying that $\widehat h_i(x_0;\lambda_i)=\widehat h_{i, \max}$, and the partial derivatives of $\widehat h_i(x;\lambda_i)$ up to order $N-1$ at $x_0$ are zero.
}

\medskip

\noindent {\bf (S2)}  { \it  $| \lambda_i|( \widehat m_{i,\max}-\widehat m_{i, \min})<\inf_{x\in \bar D_i}\int_{D_i}\kappa(y-x)dy$ in the case of $i= 2$, and $|\lambda_i|(\widehat m_{i,\max}-\widehat m_{i, \min})<1$ in the case of $i=1, 3$.
}

\medskip

\noindent{\bf (S3)} {\it  $\ds\int_{D_i}\frac{1}{\widehat h_{i,\max}-\widehat h_i(x;\lambda_i)}dx=\infty$ for $i=1, 2, 3$.}
\smallskip

\smallskip
Note that, if $|\lambda_i|\ll 1$,  then the condition (S2) is automatically satisfied for $i=1, 2, 3$.
Note also that (S1) implies (S3),  and when $i=1$ or $3$,   (S3) holds if and only if
$\ds\int_{D_i}\frac{1}{\widehat m_{i,\max}-\widehat m_i(x)}dx=\infty$.

\begin{proposition}
\label{pv-existence}
\begin{itemize}
\item[(1)] For given $\lambda_i\in\RR$,  if (S1) or (S2) or (S3) is satisfied, then $\mu_i^n(\lambda_i)$ is the principal eigenvalue of
$\mathcal{L}_i+\lambda_i m_i$.

\item[(2)] For given $\lambda_i\in\RR$, if $\mu_i^n(\lambda_i)$ is not the principal eigenvalue of $\mathcal{L}_i+\lambda_i m_i$,
then
$$
\mu_i^n(\lambda_i)=\widehat h_{i,\max}.
$$
 \end{itemize}
\end{proposition}
\begin{proof}
(1) The result that (S1) and (S2) are sufficient conditions for  $\mu_i^n(\lambda_i)$ to be the principal eigenvalue follows from the argument in Theorem B in \cite{RaSh}.
(S3) is a sufficient condition  for  $\mu_i^n(\lambda_i)$ to be the principal eigenvalue follows from
\cite[Theorem 1.1]{Cov} and Proposition \ref{spectrum-Hi}.

(2) It follows from   Proposition \ref{spectrum-Hi}.
\end{proof}

The next proposition  is essential to prove our main theorems.
Each property in the proposition  is also of independent interest.

\begin{proposition}
\label{property-PSP-prop}
Let $1\le i\le 3$.
\begin{itemize}
\item[(1)]
 Assume $m_i, m_{i, k} \in \mathcal X_i$ with $\lambda_i^km_{i,k}\to \lambda_im_i$ as $k\to\infty$  in $\mathcal{X}_i$.
 Then $\mu_i^n( \mathcal L_i+\lambda_i^k m_{i,k})\to\mu_i^n(\mathcal L_i+\lambda_i m_i)$ as $k\to\infty$.

\item[(2)] We have $\mu_i^n(0)<0$ for $i=1$, and  $\mu_i^n(0)=0$ for $i=2, 3$.

\item[(3)]  If $\lambda_i^1m_{i}^1(t, x)\leq \lambda_i^2m_{i}^2(t, x)$, then $\mu_i^n(\mathcal L_i+\lambda_i^1m_{i}^1)\leq \mu_i^n(\mathcal L_i+\lambda_i^2m_{i}^2)$. If,  in addition,
$\lambda_i^1 m_i^1(t,x)\not \equiv \lambda_i^2 m_i^2(t,x)$, then  $\mu_i^n(\mathcal L_i+\lambda_i^1m_{i}^1)< \mu_i^n(\mathcal L_i+\lambda_i^2m_{i}^2)$ .

\item[(4)] If  $\mathcal P(m_i)=\int_0^T\max_{x\in \bar D_i}m_i(t, x)dt>0$, then $\mu_i^n(\mathcal L_i+\lambda_i m_i)>0$ for $\lambda\gg1$.

\item[(5)]  The mapping  $\lambda_i\in \R^+\mapsto \mu_i^n(\mathcal L_i+\lambda_i m_i)\in \R$ is convex.

\end{itemize}
\end{proposition}

\begin{proof}

(1) follows from \cite[Proposition 3.11]{RaSh}.

(2) In the case of   $\lambda_i=0$, $\mu_i^n(0)$ is the principal eigenvalue of $\mathcal L_i$, since the condition (S2) holds.  And by Proposition  \ref{spectrum-autonomous} (1), $\mu_i^n(0)$ is the principal eigenvalue of $K_i-B_i$, associated with  an eigenfunction  $\phi_i\in  X_i^{++}$, such that
\be
\label{lambda1=0}
\int_D \kappa(y-x)\phi_1( y)dy-\phi_1( x)=\mu_1^n(0)\phi_1( x)
\ee
in the Dirichlet boundary condition case,
\be
\label{lambda2=0}
\int_D \kappa(y-x)\phi_2( y)dy-\int_Dk(y-x)dy\phi_2(x)=\mu_2^n(0)\phi_2( x)
\ee
in the Neumann boundary condition case, and
\be
\label{lambda3=0}
\int_{\R^N} \kappa(y-x)\phi_3( y)dy-\phi_3(x)=\mu_3^n(0)\phi_3(x)
\ee
in the periodic boundary condition case. In the Neumann and periodic boundary condition cases,  we have that $(\mu_2^n(0), \phi_2)=(0, 1)$ and $(\mu_3^n(0), \phi_3)=(0, 1)$ are eigenpairs, respectively.  Hence $\mu_i^n(0)=0$ for $i=2,3$. In the Dirichlet boundary case, we have
\begin{align*}
\mu_1^n(0)\int_D\phi_1^2(x)dx&=\int_D\int_D\kappa(y-x)\phi_1(y)\phi_1(x)dydx \phi_1-\int_D \phi_1^2(x)dx\\
&\le \int_D\int_D \kappa(y-x)\phi_1(y)\phi_1(x)dydx-\int_D\int_D \kappa(y-x)\phi_1^2(x)dydx\\
&=\int_D\int_D \kappa(y-x)\Big[\phi_1(y)\phi_1(x)-\frac{\phi_1^2(x)+\phi_1^2(y)}{2}\Big]dydx\\
&=-\frac12\int_D\int_D \kappa(y-x)(\phi_1(x)-\phi_1(y))^2dydx.
\end{align*}
This implies that
$\mu_1^n(0)<0$. Hence the proof is complete.

(3) Suppose that $m_{i}^1, m_{i}^2\in \mathcal X_i$, and $\lambda_2^1m_{i}^1\leq\lambda_i^2 m_{i}^2$. By Proposition \ref{comparison} (3), for any $u_0\in X_i^+$ and $t\geq s$,
\[
\Phi_i(t, s; \lambda_i^1m_{i}^1)u_0\leq \Phi_i(t, s; \lambda_i^2m_{i, }^2).
\]
This implies that
\[
r(\Phi_i(t, s; \lambda_i^1m_{i}^1))\leq r(\Phi_i(t, s; \lambda_i^2m_{i}^2)).
\]
By Proposition \ref{spectrum-sol-operator}, we have
\[
\mu_i^n(\mathcal L_i+\lambda_i^1m_{i}^1)\leq \mu_i^n(\mathcal L_i+\lambda_i^2m_{i}^2).
\]

(4) Assume that $\int_0^T \widetilde m_i(t)dt=\int_0^T \max_{x\in D_i} m_i(t, x)dt>0$, and
 we need to show that $\mu_i^n(\lambda_i)>0$ when $\lambda_i\gg1$.  We will prove the case of $i=1$,  since other cases can be proved similarly.

By the continuity of $m_1(t,x)$, there are $\delta>0$,  $x_0,x_1,\cdots,x_{n-1}\in D$,  $0=t_0<t_1<t_2<\cdots<t_n=T$, $r_0,r_1,\cdots,r_{n-1}\in \R^+$, and $m_0,m_1,\cdots,m_{n-1}\in\RR$ such that $\cup_{i=0}^{i=n-1}B(x_i, r_i)\subset D$, and  for $i=0,1,\cdots,n-1$
$$
m_i+\delta\le m_1(t,x)\le \widetilde m_1(t) \quad {\rm for}\quad t_{i}\le t\le t_{i+1},\,\, x\in B(x_i,r_i),
$$
and
$$
m_0(t_1-t_0)+m_1(t_2-t_1)+\cdots m_{n-1}(t_n-t_{n-1})>0.
$$
Let $\bar u(t,x)$ be the function defined as follows: for $ t_0\le t<t_1$,
$$
\bar u(t,x)=\begin{cases} e^{\lambda m_0 t}\quad &{\rm for}\,\, x\in B(x_0,r_0),\cr
0\quad &{\rm for}\,\, x\in\bar D\setminus B(x_0,r_0),
\end{cases}
$$
 for $ t_i\le t<t_{i+1} (i=1,2,\cdots,n-1)$,
$$
\bar u(t,x)=\begin{cases} e^{\lambda [m_0(t_1-t_0)+m_1(t_2-t_1)+\cdots m_{i-1}(t_i-t_{i-1})]+\lambda m_it}\quad &{\rm for }\,\, x\in B(x_i,r_i),\cr
0\quad &{\rm for } \,\ x\in \bar D\backslash B(x_i,r_i),
\end{cases}
$$
and for $0\le t<T$ and $k=1,2,\cdots$,
$$
\bar u(kT+t,x)=\bar u(kT,x)\bar u(t,x),\quad x\in\bar D.
$$
 We then have that for any $x\in\Bar D$, $\bar u(t,x)$ is differentiable in $t$  at all $t$ but $t_i+kT$ for $i=0,1,2,\cdots,n-1$ and $k=0,1,2,\cdots$. Moreover, for $\lambda\gg 1$, any  $t_i+kT< t<t_{i+1}+kT$,  and any $x\in\bar\Omega$,  we have
\begin{align*}
&\bar u_t(t,x)-\left[\int_D \kappa(y-x)\bar u(t,y)dy-\bar u(t,x)+\lambda_1 m_1(t,x)\bar u(t,x)\right]\\
&\le \bar u(t,x)\lambda_1 m_i-[-\bar u(t,x)+ \lambda_1 m_1(t,x) \bar u(t,x)]\\
&=\bar u(t,x)\lambda_1[m_i-m_1(t,x)]+\bar u(t,x)\\
&=\bar u(t,x)[\lambda(m_i-m(t,x))+1]\\
&\le 0
\end{align*}
for $\lambda\gg 1$,
where $i=0,1,\cdots,n-1$, $k=0,1,\cdots$.
Then by Proposition \ref{comparison},  for $\lambda\gg 1$,  $t\ge 0$ and $x\in\bar D$, we have
$$
u(t,x;u_0)\ge \bar u(t,x)
$$
for any $u_0\in X_1^+$ with  $u_0(x)\ge \bar u(0,x)$.
 For $\lambda\gg 1$, it then follows  from Proposition \ref{spectrum-limit}
that
\begin{align*}
\mu_1^n(\lambda_1)&\ge \lim_{t\to\infty}\frac{\ln\|\bar u(t,\cdot)\|_\infty}{t}\\
&=\frac{m_0(t_1-t_0)+m_1(t_2-t_1)+\cdots+m_{n-1}(t_n-t_{n-1})}{T}\\
&>0.
\end{align*}
The proof is thus complete.

(5) It suffices to  show that
for any $0\le\lambda_i^1<\lambda_i^2$,
\be
\label{convexity}
\mu_i^n\left(\frac{\lambda_i^1+\lambda_i^2}{2}\right)\le \frac{\mu_i^n(\lambda_i^1)+\mu_i^n(\lambda_i^2)}{2}.
\ee

Fix $0\le\lambda_i^1< \lambda_i^2$. By Lemma \ref{technical-lm} and Proposition \ref{pv-existence}, there are
$m_i^{1,k},m_i^{2,k}\in \mathcal{X}_i$  for $k=1,2,\cdots$ such that
$$
m_i(t,x)\le m_i^{j,k}(t,x)\quad {\rm for}\,\, j=1,2,\,\, k=1,2,3,\cdots,
$$
$$
m_i^{j,k}\to m_i\quad {\rm as}\,\, k\to\infty
$$
in $\mathcal{X}_i$ ($j=1,2$), and $\mu_i^n(\mathcal{L}_i+\lambda_i^j m_i^{j,k})$ is the principal eigenvalue of
$\mathcal{L}_i+\lambda_i^jm_i^{j,k}$ for $j=1,2$ and $k=1,2,\cdots$. It then suffices to prove
\begin{equation}
\label{aux-convexity}
\mu_i^n
\left(\frac{\lambda_i^1+\lambda_i^2}{2}\right)\le\frac{\mu_i^n(\mathcal{L}_i+\lambda_i^1m_i^{1,k})+\mu_i^n(\mathcal{L}_i+
\lambda_i^2m_i^{2,k})}{2}.
\end{equation}

Fix $k\ge 1$. Let $\mu_i^{n,j,k}=\mu_i^n(
\mathcal{L}_i+\lambda_i^j m_i^{j,k})$ ($j=1,2$).
Suppose that $\phi_i^{j,k}$ are positive eigenfunctions  of $\mathcal{L}_i+\lambda_i^j m_i^{j,k}$ corresponding to $\mu_i^{n,j,k}$. Let $\phi_i^k=\sqrt {\phi_i^{1,k}\phi_i^{2,k}}$. Note that
$\ds u_i^{1,k}(t,x)=e^{\mu_i^{n,j,k} t}\phi_i^{1,k}(t,x)$ is a solution of
$$
u_t=\int_{D} \kappa(y-x)u(t,y)dy-b_i(x) u(t,x)+\lambda_i^1 m_i^{1,k}(t, x) u(t,x),
$$
and $\ds u_i^{2,k}(t,x)=e^{\mu_i^{n,j,k}t}\phi_i^{2,k}(t,x)$ is a solution of
$$
u_t=\int_{D} \kappa(y-x)u(t,y)dy-b_i(x)u(t,x)+\lambda_i^2 m_i^{2,k}(t,x)u(t,x),
$$
where $D=\RR^N$ in the periodic boundary condition case.
Let
\[ u_i^{1,2,k}(t,x)=e^{\frac{\mu_i^{n,1,k}+\mu_i^{n,2,k}}{2}t}\sqrt{\phi_i^{1,k}(t,x)\phi_i^{2,k}(t,x)}.\]
Then, we have
\begin{align*}
&\p_t  u_i^{1,2,k}(t,x)\\
=&\frac{\mu_i^{n,1,k}+\mu_i^{n,2,k}}{2} e^{\frac{\mu_i^{n,1,k}+\mu_i^{n,2,k}}{2}t}\sqrt{\phi_i^{1,k}(t,x)\phi_i^{2,k}(t,x)}\\
& +e^{\frac{\mu_i^{n,1,k}+\mu_i^{n,2,k}}{2}t}\frac{\p_t \phi_i^{1,k}(t,x)\phi_i^{2,k}(t,x)+\phi_i^{1,k}(t,x)\p_t\phi_i^{2,k}(t,x)}{2\sqrt {\phi_i^{1,k}(t,x)\phi_i^{2,k}(t,x)}}\\
=&e^{\frac{\mu_i^{n,1,k}+\mu_i^{n,2,k}}{2}t}\frac{\int_D \kappa(y-x)[ \phi_i^{1,k}(t,y)\phi_i^{2,k}(t,x)+\phi_i^{1,k}(t,x)\phi_i^{2,k}(t,y)]dy}{
2\sqrt{\phi_i^{1,k}(t,x)\phi_i^{2,k}(t,x)}}\\
& +e^{\frac{\mu_i^{n,1,k}+\mu_i^{n,2,k}}{2}t}\frac{
-2b_i\phi_i^{1,k}(t,x)\phi_i^{2,k}(t,x)+(\lambda_i^1 m_i^{1,k} +\lambda_i^{2}m_i^{2,k})\phi_i^{1,k}(t,x)
\phi_i^{2,k}(t,x)}{2\sqrt{\phi_i^{1,k}(t,x)\phi_i^{2,k}(t,x)}}\\
\ge &e^{\frac{\mu_i^{n,1,k}+\mu_i^{n,2,k}}{2}t}\frac{\int_D \kappa(y-x)\left(2\sqrt{ \phi_i^{1,k}(t,y)\phi_i^{2,k}(t,x)}\sqrt{\phi_i^{1,k}(t,x)\phi_i^{2,k}(t,y)}\right)
dy}{
2\sqrt{\phi_i^{1,k}(t,x)\phi_i^{2,k}(t,x)}}\\
& -b_i(x)u_i^{1,2,k}(t,x)+\frac{\lambda_i^1+\lambda_i^2}{2} m_i(t,x)u_i^{1,2, k}(t,x)\\
\ge& \int_D \kappa(y-x)u_i^{1,2, k}(t,y)dy-b_i(x)u_i^{1,2, k}(t,x)+\frac{\lambda_i^1+\lambda_i^2}{2}m_i(t,x)u_i^{1,2,k}(t,x).
\end{align*}
Therefore, $u_i^{1,2,k}$ is a positive super-solution of
$$
u_t=\int_D \kappa (y-x)u(t,y)dy-b_i(x)u(t,x)+\frac{\lambda_i^1+\lambda_i^2}{2}m_i(t,x)u(t,x).
$$
This together with Propositions \ref{comparison} and  \ref{spectrum-limit} implies  that
\begin{align*}
\mu_i^n\left(\frac{\lambda_i^1+\lambda_i^2}{2}\right)\le \lim_{t\to\infty} \frac{\ln u_i^{1,2,k}(t,x)}{t}=\frac{\mu_i^{n,1,k}+\mu_i^{n,2,k}}{2}.
\end{align*}
This proves \eqref{aux-convexity}. Letting $k\to \infty$.  we get \eqref{convexity}.
\end{proof}

\begin{corollary}
\label{strict-convex}
Assume that $m_i(t,x)\not\equiv m_i(t)$ and $0\le \lambda_1^1<\lambda_i^2$.
If  $\mu_i^n(\mathcal L_i+\lambda_i^{j} m_i)$ ($j=1,2$) are the principal eigenvalue of $\mathcal L_i+\lambda_i^{j} m_i$,
then
$$
\mu_i^n\left(\frac{\lambda_i^1+\lambda_i^2}{2}\right)<\frac{\mu_i^n(\lambda_i^1)+\mu_i^n(\lambda_i^2)}{2}.
$$
\end{corollary}

\begin{proof}
We denote the principal eigenvalues of  $\mathcal{L}_i+\lambda_i^1 m_i$ and $\mathcal{L}_i+\lambda_i^2m_i$ by $\mu_i^n(\lambda_i^1)$ and  $\mu_i^n(\lambda_i^2)$,
 respectively. Let $\phi_i^j$ be a positive eigenfunction of $\mathcal{L}_i+\lambda_i^j m_i$ ($j=1,2$).
By the assumption that $m_i(t,x)\not \equiv m_i(t)$, we have $\frac{\phi_i^1(t,x)}{\phi_i^2(t,x)}\not\equiv$constant.
In fact, if $\phi_i^1(t,x)=c\phi_i^2(t,x)$ for some $c>0$, then we have
$$
\mathcal{L}_i \phi_i^1(t,x)+\lambda_i^1 m_i(t,x)\phi_i^1(t,x)=\mu_i^n(\lambda_i^1)\phi_i^1(t,x)
$$
and
$$
\mathcal{L}_i\phi_i^1(t,x)+\lambda_i^2 m_i(t,x)\phi_i^1(t,x)=\mu_i^n(\lambda_i^2)\phi_i^1(t,x).
$$
It then follows that
$$
\lambda_i^1 m_i(t,x)\phi_i^1(t,x)-\mu_i^n(\lambda_i^1)\phi_i^1(t,x)=
\lambda_i^2 m_i(t,x)\phi_i^1(t,x)-\mu_i^n(\lambda_i^2)\phi_i^1(t,x).
$$
This implies that
$$
m_i(t,x)=\frac{\mu_i^n(\lambda_i^1)-\mu_i^n(\lambda_i^2)}{\lambda_i^1-\lambda_i^2}\equiv {\rm constant}
$$
This is a contradiction. Hence $\frac{\phi_i^1(t,x)}{\phi_i^2(t,x)}\not\equiv$constant.

We then have
$$
\frac{\phi_i^1(t,y)\phi_i^2(t,x)+\phi_i^1(t,y)\phi_i^2(t,x)}{2\sqrt{\phi_i^1(t,x)\phi_i^2(t,x)}}>\sqrt{\phi_i^1(t,y)\phi_i^2(t,y)}.
$$
By the arguments of
Proposition \ref{property-PSP-prop},  we have
$$
\mu_i^n\left(\frac{\lambda_i^1+\lambda_i^2}{2}\right)<\frac{\mu_i^n(\lambda_i^1)+\mu_i^n(\lambda_i^2)}{2}.
$$
The corollary is thus proved.
\end{proof}

\section{Proofs of the main results}

In this section,   we prove our main results.

\subsection{Dirichlet boundary condition  case}

In this subsection, we prove  Theorem \ref{PSP-diri}  and Corollary \ref{PSP-diri-coro}.

\begin{proof} [Proof of Theorem \ref{PSP-diri}]
(1) We first assume that there is $\lambda_1^p>0$ such that $\mu_1^n(\lambda_1^p)=0$, and  prove that (D) holds, that is
$\int_0^T\widetilde m_1(t)dt=\int_0^T \max_{x\in \bar {D_1}} m_1(t, x)dt>0$.

 Assume that $\int_0^T\widetilde m_1(t)dt\le 0$. For any  $\lambda_1>0$, we use $\widetilde \mu_1^n(\lambda_1)$ to denote the principal spectrum point of  the following eigenvalue problem
$$
\begin{cases}
-\p_t u+\int_{D_1} \kappa(y-x)u(t,y)dy-u(t,x)+\lambda_1 \widetilde m_1(t)u=\widetilde\mu_1 u\quad \text{ in }\bar D,\cr
u(t+T,x)=u(t,x).
\end{cases}
$$
 By Proposition \ref{spectrum-autonomous} (2),  we have
 $$
\widetilde \mu_1^n(\lambda_1)=\mu_i^n(0)+\frac{\lambda_1}{T}\int_0^T \widetilde m(t)dt,
$$
where $\mu_1^n(0)<0$ by Proposition \ref{property-PSP-prop}(2). And by Proposition \ref{property-PSP-prop}(3),  we have
$$
\mu_1^n(\lambda_1)\le \widetilde\mu_1^n(\lambda_1),
$$
since $m_1(t, x)\leq \widetilde m_1(t)$ and $\lambda_1>0$.
Hence $\mu_1^n(\lambda_1)<0$ for any $\lambda\ge 0$. This is a
contradiction. Hence $\int_0^T\widetilde m_1(t)dt>0$.

Next, we assume that (D) holds, and prove that there is a unique $\lambda_1^p>0$ such that
$\mu_1^n(\lambda_1^p)=0$.

According to Proposition \ref{property-PSP-prop}(2),  $\mu_1^n(\lambda_1)<0$ for $\lambda_1=0$.  Meanwhile, from Proposition \ref{property-PSP-prop}(4), we have $\mu_1^n(\lambda_1)>0$ for $\lambda\gg1$. Thus
there is $\lambda_1^p>0$ such that $\mu_1^n(\lambda_1^p)=0$.

(2) Suppose there exist $0<\lambda_1^{p, 1}\leq \lambda_1^{p, 2}$ such that
$\mu_1^n(\lambda_1^{p, 1})=\mu_1^n(\lambda_1^{p, 2})=0$. We need to show that $\lambda_1^{p, 1}=\lambda_1^{p, 2}$.
Assume that $\lambda_1^{p,1}<\lambda_1^{p,2}$.
By the convexity of $\mu_1^n(\lambda_1)$, $\mu_1^n(\lambda_1)=0$ for $\lambda_1^{p,1}\le \lambda_1\le\lambda_1^{p,2}$.

If $m_1(t,x)\equiv m_1(t)$, we have $\mu_1^n(\lambda_1^{p,1})=\mu_1^n(0)+\lambda_1^{p,1}\widehat m_1=0$
and
 $\mu_1^n(\lambda_1^{p,2})=\mu_1^n(0)+\lambda_1^{p,2}\widehat m_1=0$. Note that $\mu_1^n(0)<0$.  We then must have $\widehat m_1>0$ and then
$$
 0=\mu_1^n(\lambda_1^{p,1})=\mu_1^n(0)+\lambda_1^{p,1}\widehat m_1<
\mu_1^n(\lambda_1^{p,2})=\mu_1^n(0)+\lambda_1^{p,2}\widehat m_1=0.
$$
This is a contradiction. Hence $\lambda_1^{p,1}=\lambda_1^{p,2}$.

Suppose that $m_1(t,x)\not\equiv m_1(t)$.
Assume that $\mu_1^n(\lambda_1^{p,1})$ is the principal eigenvalue of $\mathcal{L}_1+\lambda_1^{p,1}m_1$. Then for
$\lambda_1>\lambda_1^{p,1}$ with $\lambda_1-\lambda_1^{p,1}\ll 1$, $\mu_1^n(\lambda_1)$ is also the
principal eigenvalue of $\mathcal{L}_1+\lambda_1 m_1$.
By Corollary \ref{strict-convex}, we have $\mu_1^n(\lambda_1)<0$ for $\lambda_1$ with $\lambda_1-\lambda_1^{p,1}\ll 1$.
This is a contradiction.
Hence  $\mu_1^n(\lambda_1^{p,1})$ is not the principal eigenvalue of $\mathcal{L}_1+\lambda_1^{p,1}m_1$.
Similarly,  $\mu_1^n(\lambda_1^{p,2})$ is not the principal eigenvalue of $\mathcal{L}_1+\lambda_1^{p,2}m_1$.
We then have that  $\mu_1^n(\lambda_1^{p, 1})$ and
 $\mu_1^n( \lambda_1^{p, 2})$ are not  eigenvalues of $\mathcal L_1+\lambda_1^{p,1}m_1$ and
$\mathcal{L}_1+\lambda_1^{p,2}m_1$, respectively.  By Proposition \ref{pv-existence} (2), we have
\[
\mu_1^n(\lambda_1^{p,1})=h_{1,\max}(\lambda_1^{p,1})= -1+\lambda_1^{p, 1}m_{1,\max},
\]
 and
\[\mu_1^n(\lambda_1^{p,2})=h_{1,\max}(\lambda_1^{p,2})=-1+\lambda_1^{p, 2}m_{1,\max}.
\]
It thus follows that    $\lambda_1^{p, 1}=\lambda_1^{p, 2}$ if $\mu_1^n(\lambda_1^{p, 1})=\mu_1^n(\lambda_1^{p, 2})=0$.
This is also  a contradiction.

Therefore, $\lambda_1^{p,1}=\lambda_1^{p,2}$.
\end{proof}

\begin{proof}[Proof of Corollary \ref{PSP-diri-coro}]
(1) Assume that $m_1(t,x)\equiv m_1(t)$. Then  $ \mathcal{P}(m_1)=\int_0^T m_1(t)dt$ and
$$
\mu_1^n(\mathcal{L}_1+\lambda_1 m_1)=\mu_1^n(0)+\lambda_1 \widehat m_1.
$$
(1) then follows from Theorem \ref{PSP-diri}.

(2) Assume that $m_1(t,x)\equiv m_1(x)$. Then $\mathcal{P}(m_1)=T\cdot \max_{x\in\bar D_1}m_1(x)$. Then (2)  follows from
Theorem \ref{PSP-diri} and Proposition \ref{spectrum-autonomous}.
\end{proof}

\subsection{Neumann boundary condition case}

In this subsection, we prove  Theorem \ref{PSP-neum}  and Corollary \ref{PSP-neum-coro}.

\begin{proof}[Proof of Theorem \ref{PSP-neum}] We first prove (2).

Suppose that there exist $0<\lambda_2^{p, 1}\leq \lambda_2^{p, 2}$ such that $\mu_2^n(\lambda_2^{p, 1})=\mu_2^n(\lambda_2^{p, 2})=0$. We need to show that $\lambda_2^{p, 1}= \lambda_2^{p, 2}$.

Assume that $\lambda_2^{p,1}<\lambda_2^{p,2}$. By the convexity of $\mu_2^n(\lambda_2)$,
$\mu_2^n(\lambda_2)=0$ for
$\lambda_2^{p,1}\le \lambda_2\le\lambda_2^{p,2}$.

By the similar argument as in the Dirichlet boundary case,
  $\mu_2^n(\lambda_2^{p,1})$ and  $\mu_2^n(\lambda_2^{p,2})$ are not  eigenvalues of $\mathcal{L}_2+\lambda_2^{p,1}m_2$ and
 $\mathcal{L}_2+\lambda_2^{p,2}m_2$, respectively (the assumption $m_2(t,x)\not\equiv m_2(t)$ is used here).
 By Proposition \ref{pv-existence} (2), we have
\[
\mu_2^n(\lambda)=\widehat h_{2,\max}(\lambda)=\max_{x\in  {\bar D_2}}\left[ -b_2(x)+\lambda\widehat m_{2}(x)\right]=0 \quad \text{for all}\,\, \lambda\in[\lambda_2^{p,1},\lambda_2^{p,2}],
\]
where $b_2(x)=\int_D\kappa(y-x)dy$. Let $x_\lambda\in \bar D_2$ be such that
\be
\label{proof-neum-unique}
\max_{x\in \bar D_2}\left[ -b_2(x)+\lambda\widehat m_{2}(x)\right]=-b_2(x_\lambda)+\lambda \widehat m_2(x_\lambda)=\mu_2^n(\lambda)=0
\ee
for $\lambda\in [\lambda_2^{p,1},\lambda_2^{p,2}]$.  Note that $-b_2(x)<0$ for any $x\in\bar D_2$. Hence
we must have $ \widehat m_2(x_\lambda)>0$.  In particular, $\widehat m_2(x_{\lambda_2^{p,1}})>0$. We then have
$$
-b_2(x_{\lambda_2^{p,1}})+\lambda_2^{p,1}\widehat m_2(x_{\lambda_2^{p,1}})<-b_2(x_{\lambda_2^{p,1}})+\lambda_2^{p,2}
\widehat m_2(x_{\lambda_2^{p,1}})\le -b_2(x_{\lambda_2^{p,2}})+\lambda_2^{p,2}
\widehat m_2(x_{\lambda_2^{p,2}}),
$$
which contradicts to \eqref{proof-neum-unique}.

It thus follows that    $\lambda_2^{p, 1}=\lambda_2^{p, 2}$ if $\mu_2^n(\lambda_2^{p, 1})=\mu_2^n(\lambda_2^{p, 2})=0$.

Next, we prove (1).

 First
suppose that there is $\lambda_2^p>0$ such that $\mu_2^n(\lambda_2^p)=0$. We prove that
$\int_0^T \widetilde m_2(t)dt>0$ and $\int_D\widehat m_2(x)dx<0$ with $\widetilde m_2(t)=\max_{x\in \bar D_2}m_2(t, x)$ and $\widehat m_2(x)=\frac1T\int_0^Tm_2(t, x)dt$.

Assume that $\int_0^T \widetilde m_2(t)dt\le 0$. Note that $m_2(t,x)\le \widetilde m_2(t)$ and $m_2(t,x)\not\equiv \tilde m_2(t)$. Hence, by Proposition \ref{spectrum-autonomous}(2) and Proposition \ref{property-PSP-prop}(3), we have
$$
\mu_2^n(\lambda_2)< \lambda_2\frac{\int_0^T\widetilde m_2(t)dt}{T}\le 0
$$
for $\lambda_2> 0$.
This is a contradiction. Therefore $\int_0^T\widetilde m_2(t)dt>0$.

  Proposition \ref{property-PSP-prop}, together with  $\mu_2^n(\lambda_2^p)=0$ and (2), leads to $\mu_2^n(\lambda_2)<0$ for $0<\lambda_2\ll 1$.
Note that for $|\lambda_2|\ll 1$, $\mu_2^n(\lambda_2)$ is a principal eigenvalue, associated with a positive eigenfunction, denoted by $\phi_2^{\lambda_2}(t,x)$. Then
$$
-\frac{\p_t\phi_2^{\lambda_2}(t,x)}{\phi_2^{\lambda_2}(t,x)}+\int_{D_2} \kappa(y-x)\frac{\phi_2^{\lambda_2}(t,y)-\phi_2^{\lambda_2}(t,x)}{\phi_2^{\lambda_2}(t,x)}dy
+\lambda_2 m_2(t,x)=\mu_2^n(\lambda_2).
$$
Hence for $0<\lambda_2\ll 1$,
\begin{align*}
&\mu_2^n(\lambda_2)\cdot T\cdot |D_2| -\lambda_2 T \int_{D_2} \widehat m_2(x)dx\\
&=
\int_0^T \int_{D_2}\int_{D_2} \kappa(y-x)\frac{\phi_2^{\lambda_2}(t,y)-\phi_2^{\lambda_2}(t,x)}{\phi_2^{\lambda_2}(t,x)}dydxdt\\
&=\frac{1}{2}\int_0^T \int_{D_2}\int_{D_2} \kappa(y-x)\Big[\frac{\phi_2^{\lambda_2}(t,y)-\phi_2^{\lambda_2}(t,x)}{\phi_2^{\lambda_2}(t,x)}
+\frac{\phi_2^{\lambda_2}(t,x)-\phi_2^{\lambda_2}(t,y)}{\phi_2^{\lambda_2}(t,y)}\Big]dydxdt\\
&=\frac{1}{2}\int_0^T\int_{D_2}\int_{D_2} \kappa(y-x)\frac{(\phi_2^{\lambda_2}(t,y)-\phi_2^{\lambda_2}(t,x))^2}{\phi_2^{\lambda_2}(t,x)\phi_2^{\lambda_2}(t,y)}dydxdt.
\end{align*}
Since $m_2(t,x)\not \equiv m_2(t)$, $\phi_2^{\lambda_2}(t,y)\not\equiv \phi_2^{\lambda_2}(t,x)$. We then  have
$$
0>\mu_2^n(\lambda_2)\cdot T\cdot |D_2|> \lambda_2 T\int_{D_2}\widehat m_2(x)dx \quad \text{ for } 0<\lambda_2\ll 1,
$$
and then $\int_{D_2} \widehat m_2(x) dx<0$.

Next, suppose that $\int_0^T\widetilde m_2(t)dt>0$ and $\int_{D_2} \widehat m_2(x)dx<0$. We prove that there is $\lambda_2^p>0$ such that
$\mu_2^n(\lambda_2^p)=0$.

By Proposition \ref{property-PSP-prop} (4),
$\mu_2^n(\lambda_2)>0$ for $\lambda_2\gg 1$.
Note that $\mu_2(0)=0$ and for $0<\lambda_2\ll 1$, $\mu_2^n(\lambda_2)$ is a principal eigenvalue of \eqref{neumann-ori}.
Suppose that $\phi_2^{\lambda_2}(t,x)$ is a positive principal eigenfunction with $\|\phi_2^{\lambda_2}(\cdot,\cdot)\|_\infty=1$.  Note that $\phi_2^{\lambda_2}(t,x)$ is differentiable in $\lambda_2$
and $\phi_2^0(t,x)\equiv 1$.  Hence
$$
\phi_2^{\lambda_2}(t,y)-\phi_2^{\lambda_2}(t,x)=\phi_2^{\lambda_2}(t,y)-1-(\phi_2^{\lambda_2}(t,x)-1)=O(\lambda_2).
$$
Then by the above arguments,
\begin{align*}
&\mu_2^n(\lambda_2)\cdot T\cdot|D_2|\\
=&\lambda_2\int_{D_2} \widehat m_2(x)dx+\frac{1}{2}\int_0^T\int_{D_2}\int_{D_2} \kappa(y-x)\frac{(\phi_2^{\lambda_2}(t,y)-\phi_2^{\lambda_2}(t,x))^2}{\phi_2^{\lambda_2}(t,x)\phi_2^{\lambda_2}(t,y)}dydxdt\\
=&\lambda_2\int_{D_2}\widehat m_2(x)dx+O(\lambda_2^2).
\end{align*}
It then follows that $\mu_2^n(\lambda_2)<0$ for $0<\lambda_2\ll 1$.
Then there is $\lambda_2^p>0$ such that $\mu_2^n(\lambda_2^p)=0$.
\end{proof}

\begin{proof}[Proof of Corollary \ref{PSP-neum-coro}]
(1) Assume that $m_2(t,x)\equiv m_2(t)$. Then  by Proposition \ref{spectrum-autonomous} (2) and \ref{property-PSP-prop} (2), we have
$$
\mu_2^n(\mathcal{L}_2+\lambda_2 m_2)=\mu_2^n(0)+\lambda_2 \widehat m_2=\lambda_2\widehat m_2.
$$
It then follows that if $\widehat m_2\not =0$, then there is no positive principal spectrum point of $\mathcal{L}_2+\lambda_2 m_2$.
If $\widehat m_2=0$, then every positive $\lambda$ is a principal spectrum point of $\mathcal{L}_2+\lambda_2m_2$.

(2) Assume that $m_2(t,x)\equiv m_2(x)$. Then $\mathcal{P}(m_2)=T\cdot \max_{x\in\bar D_2}m_2(x)$. (2) then follows from
Theorem \ref{PSP-neum} and Proposition \ref{spectrum-autonomous}.

\end{proof}

\subsection{Periodic boundary condition case}

In this subsection, we prove  Theorem \ref{PSP-peri}  and Corollary \ref{PSP-peri-coro}

\begin{proof}[Proof of Theorem \ref{PSP-peri}]

First,  (2)  can proved by the similar arguments as those in Theorem \ref{PSP-neum}(2).

We prove (1) in the following.

Suppose that there is $\lambda_3^p>0$ such that $\mu_3^n(\lambda_3^p)=0$. We prove that
$\int_0^T \widetilde m_3(t)dt>0$ and $\int_{D_3}\widehat m_3(x)dx<0$.

Assume that $\int_0^T \widetilde m_3(t)dt\le 0$. Note that $m_3(t,x)\le \widetilde m_3(t)$ and $m_3(t,x)\not\equiv \widetilde m_3(t)$. Hence
$$
\mu_3^n(\lambda_3)<\frac{\lambda_3}T\int_0^T\widetilde m_3(t)dt\le 0
$$
for $\lambda_3> 0$.
This is a contradiction. Therefore $\int_0^T\widetilde m_3(t)dt>0$.

 Proposition \ref{property-PSP-prop}, together with $\mu_3^n(\lambda_3^p)=0$,
leads to $\mu_3^n(\lambda_3)<0$ for $0<\lambda_3\ll 1$.
Note that for $|\lambda_3|\ll 1$, $\mu_3^n(\lambda_3)$ is a principal eigenvalue. Suppose that $\phi_3^{\lambda_3}(t,x)$ is a positive principal eigenfunction. Then
$$
-\frac{\p_t\phi_3^{\lambda_3}(t,x)}{\phi_3^{\lambda_3}(t,x)}+\int_{\RR^N} \kappa(y-x)\frac{\phi_3^{\lambda_3}(t,y)-\phi_3^{\lambda_3}(t,x)}{\phi_3^{\lambda_3}(t,x)}dy
+\lambda_3 m_3(t,x)=\mu_3^n(\lambda_3).
$$
Let
$$
\widetilde \kappa(z)=\sum_{(k_1,k_2,\cdots,k_N)\in\ZZ^N}\kappa(z+(k_1p_1,k_2p_2,\cdots,k_Np_N)).
$$
Then $\widetilde \kappa(-z)=\widetilde \kappa(z)$ and
$$
\int_{\RR^N}\kappa(y-x)\phi_3^{\lambda_3}(t,y)dy=\int_{D_3}\widetilde \kappa(y-x) \phi_3^{\lambda_3}(t,y)dy.
$$
Hence for $0<\lambda_3\ll 1$,
\begin{align*}
&\mu_3^n(\lambda_3)\cdot T\cdot |D_3| -\lambda_3T \int_{D_3} \widehat m_3(x)dx\\
&=
\int_0^T \int_{D_3}\int_{D_3}  \widetilde\kappa(y-x)\frac{\phi_3^{\lambda_3}(t,y)-\phi_3^{\lambda_3}(t,x)}{\phi_3^{\lambda_3}(t,x)}dydxdt\\
&=\frac{1}{2}\int_0^T \int_{D_3}\int_{D_3}\widetilde \kappa(y-x)\Big[\frac{\phi_3^{\lambda_3}(t,y)-\phi_3^{\lambda_3}(t,x)}{\phi_3^{\lambda_3}(t,x)}
+\frac{\phi_3^{\lambda_3}(t,x)-\phi_3^{\lambda_3}(t,y)}{\phi_3^{\lambda_3}(t,y)}\Big]dydxdt\\
&=\frac{1}{2}\int_0^T\int_{D_3}\int_{D_3}\widetilde \kappa(y-x)\frac{(\phi_3^{\lambda_3}(t,y)-\phi_3^{\lambda_3}(t,x))^2}{\phi_3^{\lambda_3}(t,x)\phi_3^{\lambda_3}(t,y)}dydxdt.
\end{align*}
Since $m_3(t,x)\not \equiv m_3(t)$, $\phi_3^{\lambda_3}(t,y)\not\equiv \phi_3^{\lambda_3}(t,x)$. Hence for $0<\lambda_3\ll 1$,
$$
0>\mu_3^n(\lambda_3)\cdot T\cdot |D_3|> \lambda_3 T\int_{D_3}\widehat m_3(x)dx
$$
and then $\int_{D_3} \widehat m_3(x) dx<0$.

Next, suppose that $\int_0^T\widetilde m_3(t)dt>0$ and $\int_{D_3} \widehat m_3(x)dx<0$. We prove that there is $\lambda_3^p>0$ such that
$\mu_3^n(\lambda_3^p)=0$.

By Proposition \ref{property-PSP-prop} (4),
$\mu_3^n(\lambda_3)>0$ for $\lambda_3\gg 1$.
Note that $\mu_3^n(0)=0$ and for $0<\lambda_3\ll 1$, $\mu_3^n(\lambda_3)$ is a principal eigenvalue of \eqref{periodic-ori}.
Suppose that $\phi_3^{\lambda_3}(t,x)$ is a positive principal eigenfunction with $\|\phi_3^{\lambda_3}(\cdot,\cdot)\|_\infty=1$.  Note that $\phi_3^{\lambda_3}(t,x)$ is differentiable in $\lambda_3$
and $\phi_3^0(t,x)\equiv 1$.  Hence
$$
\phi_3^{\lambda_3}(t,y)-\phi_3^{\lambda_3}(t,x)=\phi_3^{\lambda_3}(t,y)-1-(\phi_3^{\lambda_3}(t,x)-1)=O(\lambda_3).
$$
Then by the above arguments,
\begin{align*}
&\mu_3^n(\lambda_3)\cdot T\cdot|D_3|\\
=&\lambda_3\int_{D_3} \widehat m_3(x)dx+\frac{1}{2}\int_0^T\int_{D_3}\int_{D_3} \widetilde\kappa(y-x)\frac{(\phi_3^{\lambda_3}(t,y)-\phi_3^{\lambda_3}(t,x))^2}{\phi_3^{\lambda_3}(t,x)\phi_3^{\lambda_3}(t,y)}dydxdt\\
=&\lambda_3\int_{D_3}\widehat m_3(x)dx+O(\lambda_3^2).
\end{align*}
It then follows that $\mu_3^n(\lambda_3)<0$ for $0<\lambda_3\ll 1$.
Then there is $\lambda_3^p>0$ such that $\mu_3^n(\lambda_3^p)=0$.
\end{proof}

\begin{proof}[Proof of Corollary \ref{PSP-peri-coro}]
It can be proved by a similar argument as in the proof of Corollary \ref{PSP-neum-coro}.
\end{proof}

\subsection{Upper bounds and the existence of principal eigenvalues}

In this subsection, we prove Theorem \ref{upper-bounds}.

\begin{proof} [Proof of Theorem \ref{upper-bounds}]
(1) First, we prove the upper bound of \eqref{EP-lambdai} in the case of $i=1$. If   \eqref{average-weight-pev} has a unique positive principal spectrum point $\lambda_1^p(\widehat m_1)$, we need to show  that   \eqref{EP-lambdai} also has a unique positive principal spectrum point $\lambda_1^p(m_1)$, and $\lambda_1^p(m_1)\leq \lambda_1^p(\widehat m_1)$.

By Corollary \ref{PSP-diri-coro} (2), \eqref{average-weight-pev} has a unique positive principal spectrum point $ \lambda_1^p(\widehat m_1)$ if and only if
\be
\label{upper-bound-diri}
\widehat m_1(x_0)>0 \quad \text{ for some } x_0\in D_1,
\ee
where $\widehat m_1(x_0)=\frac1T\int_0^Tm_1(t, x_0)dt$. Since $\int_0^Tm_1(t, x_0)dt\leq \int_0^T\max_{x\in \bar D_1}m_1(t, x)dt$, we know that $m_1(t, x)$ satisfies (D), that is
\[
\mathcal P(m_1)=\int_0^T\max_{x\in \bar D_1}m_1(t, x)dt>0.
\]
By Theorem \ref{PSP-diri}, we know that  \eqref{EP-lambdai} also has a unique positive principal spectrum point $\lambda_1^p(m_1)$.

To prove the upper bound of $\lambda_1^p(m_1)$, denote the principal spectrum point of \eqref{average-pev} by $\mu_1^n(\lambda_1,\widehat m_1)$, and the principal spectrum point of \eqref{EP-mui}  by $\mu_1^n(\lambda_1,m_1)$.
%  It is shown in \cite[Theorem C]{RaSh} that  $\mu_1^n(\lambda_1)\geq \widehat\mu_1^n(\lambda_1)$.
We have
\[
\mu_1^n(\lambda_1^p(\widehat m_1),\widehat m_1)\leq \mu_1^n(\lambda_1^p(\widehat m_1),m_1))
\]
by Proposition \ref{property-PSP-prop} (3).
By Definition \ref{PSP-def},  we have
\[
\mu_1^n(\lambda_1^p(\widehat m_1),\widehat m_1)=\mu_1^n(\lambda_1^p(m_1),m_1)=0.
\]
Hence,
\[
\mu_1^n(\lambda_1^p(m_1),m_1)\leq \mu_1^n( \lambda_1^p(\widehat m_1),m_1).
\]
By Proposition \ref{property-PSP-prop} (3), we have
\[
\lambda_1^p(m_i)\leq \lambda_1^p(\widehat m_i).
\]
The proof in the case of $i=2$ or $i=3$ is similar to the case of $i=1$. So we omit it.

(2) (i) It follows from the fact that  the conditions in (i) imply that (S1) holds.

(ii) It follows from (1) and Proposition \ref{pv-existence}.
\end{proof}

\subsection{Applications}
In this section, we apply Theorem \ref{PSP-diri}, \ref{PSP-neum} and \ref{PSP-peri} to study the existence of positive time-periodic  solution of the KPP type equations  \eqref{n-kpp-diri}, \eqref{n-kpp-neum}, and \eqref{n-kpp-peri}.

Recall  the KPP type equation   \eqref{n-kpp-diri} (\eqref{n-kpp-neum} or \eqref{n-kpp-peri}) with $f_i(t, x, u) (i=1, 2, 3)$ satisfying (F).  The eigenvalue problems  of   \eqref{n-kpp-diri}, \eqref{n-kpp-neum}, and \eqref{n-kpp-peri} linearized at $u=0$ are
\begin{equation}
\label{KPP-diri-L}
\begin{cases}
 -\p_t u(t, x)+\left[\int_D \kappa(y-x)u(t, y)dy-u(t, x)\right]+\lambda_1 m_1(t,x)u(t, x)=\mu_1  u,\\
u(t+T, x)=u(t, x).\\
u(t, x)\geq0,
\end{cases}
\end{equation}
for the Dirichlet boundary condition,
\begin{equation}
\label{KPP-neum-L}
\begin{cases}
 -\p_t u(t, x)+\int_D \kappa(y-x)[u(t, y)-u(t, x)]dy+\lambda_2 m_2(t,x)u(t, x)=\mu_2  u,\\
u(t+T, x)=u(t, x),\\
u(t, x)\geq 0
\end{cases}
\end{equation}
for the Neumann boundary condition, and
 \begin{equation}
\label{KPP-peri-L}
\begin{cases}
 -\p_t u(t, x)+\int_{\R^N} \kappa(y-x)[u(t, y)-u(t, x)]dy+\lambda_3 m_3(t,x)u(t, x)=\mu_3  u,\\
u(t+T, x)=u(t, x+p_j{\bf e_j})=u(t, x),\\
u(t, x)\geq 0
\end{cases}
\end{equation}
for the periodic boundary condition, where $m_i(t, x)= f_i(t, x, 0)$. It is assumed that $m_1$, $m_2$, and $m_3$ satisfies (D), (N), and (P), respectively.

Denote the principal spectrum point of \eqref{KPP-diri-L}, \eqref{KPP-neum-L}, and \eqref{KPP-peri-L} by $\mu_1^n(\lambda_1)$, $\mu_2^n(\lambda_2)$, and $\mu_3^n(\lambda_3)$, respectively.
Suppose $\lambda_i^p$ is such that $\mu_i^n(\lambda_i^p)=0$, we need to show that  \eqref{n-kpp-diri} (\eqref{n-kpp-neum} or \eqref{n-kpp-peri}) admits a unique positive time-periodic solution $u^*(t, x)$ if and only if   $\lambda_i>\lambda_i^p$.
\begin{proof}[Proof of Theorem \ref{KPP}]

We only prove the case of $i=1$, and other cases ($i=2$ or $i=3$) can be shown similarly.

First, we show if $\lambda_1>\lambda_1^p$, then  \eqref{n-kpp-diri} admits a unique positive time-periodic solution $u^*(t, x)$.
By \cite[Theorem E]{RaSh}, the KPP type equation  \eqref{n-kpp-diri} has a unique positive time periodic solution,  if  $f_1$ satisfies the following monostable assumptions:

\medskip

 \noindent {\bf (H1)} {\it  $f_1$ is $C^1$ in $t\in\RR$ and  $C^3$ in $(x,u)\in \RR^N\times\RR$; $f_1(t, x, u)<0$ for $u\gg 1$ and $\partial _u f_1(t, x, u)<0$ for $u\geq0$};
 $f_1(t+T,x,u)=f_1(t,x,u)$.

\medskip

\noindent {\bf (H2)} {\it   $\mu_1^n(\lambda_1)>0$, where $\mu_1^n(\lambda_1)$ is
the principle spectrum  point of \eqref{KPP-diri-L}.
}

\medskip

The condition (H1) is exactly (F), so we only need to show that (H2) is satisfied if $\lambda_1>\lambda_1^p$. Since $m_1=f_1(t, x, 0)$ satisfies (D), we have $\mu_1^n(\lambda_1^p)=0$ by Theorem \ref{PSP-diri}. Since $\lambda_1>\lambda_1^p$, we know that $\mu_1^n(\lambda_1)>\mu_1^n(\lambda_1^p)$
  by Theorem \ref{PSP-diri} and  \ref Proposition \ref{property-PSP-prop} (3). It follows that  $\mu_1^n(\lambda_1)>0$. Hence (H2) is satisfied.

%
%
%Second, we show that  if $\lambda_1\leq \lambda_i^p$, then \eqref{n-kpp} with $i=1$ has no positive time-periodic solution. In fact, if $\lambda_1\leq \lambda_1^p$, we have $\mu_1^n(\lambda_1)\leq \mu_1^n(\lambda_1^p)$ by Proposition \ref{property-PSP-prop} (3). Since $m_1=\p_u f(t, x, 0)$ satisfies (D), $\lambda_1^p$ is the unique positive number such that $\mu_1^n(\lambda_1^p)=0$. Hence, we have $\mu_1^n(\lambda_1)\leq 0$.

Second, consider \eqref{n-kpp-diri},
and we need to show that if \eqref{n-kpp-diri} has a solution, then $\lambda_1>\lambda_1^p$, where $\lambda_1^p$ is the unique positive number such that $\mu_1^n(\lambda_1^p)=0$.

Suppose that   \eqref{n-kpp-diri} admits a positive  solution  $u^*(t, x)$.
Then
$$
-\p_t u^*(t,x)+\int_D\kappa(y-x)u^*(t,y)dy-u^*(t,x)+\lambda_1 f_1(t,x,u^*(t,x)) u^*(t,x)=0.
$$
This implies that
$\mu_1^n(\mathcal{L}_1+\lambda_1 m^*_1)=0$, where $m^*_1(t,x)=f_1(t,x,u^*(t,x))$.
Note that $f_1(t,x,0)>f_1(t,x,u^*(t,x))$. Hence $\mu_1^n(\lambda_1)>0$. Therefore,
$\lambda_1^p<\lambda_1.$
\end{proof}

\end{document}